\RequirePackage[table]{xcolor}
\documentclass[a4paper,12pt]{amsart}
\usepackage{amsfonts}
\usepackage{amsmath}
\usepackage{ifthen}
\usepackage{amsrefs}
\usepackage{amsthm}
\usepackage[all,cmtip]{xy}
\usepackage{amssymb}
\usepackage{hyperref}
\usepackage[tocflat]{tocstyle}
\usetocstyle{standard}

\usepackage{hyperref}
\usepackage{url}
\usepackage[shortlabels]{enumitem}
\usepackage[paper=a4paper,left=20mm,right=20mm,top=25mm,bottom=30mm]{geometry}

\usepackage{tikz}
\usetikzlibrary{backgrounds}
\usepackage{tkz-euclide}
\usepackage{xcolor}
\usetikzlibrary{trees,snakes,shapes.geometric}
\usepackage[outline]{contour}
\contourlength{1.5pt}
\usetikzlibrary{positioning,backgrounds}
\usetikzlibrary{%
  matrix,%
  calc,%
  arrows%
}

\usepackage{mathtools}   % loads »amsmath«
\usepackage{graphicx}

\setlist[enumerate]{topsep=0em, itemsep= -0em, parsep = 0 em, label=$(\alph*)$}
\let\emptyset\varnothing
\nocite{*}

\newcommand{\NN}{\mathbb{N}}

\newcommand{\ZZ}{\mathbb{Z}}

\newcommand{\cA}{\mathcal{A}}
\newcommand{\cB}{\mathcal{B}}
\newcommand{\cC}{\mathcal{C}}

\newcommand{\cE}{\mathcal{E}}
\newcommand{\cF}{\mathcal{F}}
\newcommand{\cO}{\mathcal{O}}

\newcommand{\cW}{\mathcal{W}}
\newcommand{\cX}{\mathcal{X}}
\newcommand{\cY}{\mathcal{Y}}

\DeclareMathOperator{\Rad}{Rad}

\DeclareMathOperator{\rep}{rep}

\DeclareMathOperator{\re}{re}

\DeclareMathOperator{\ql}{ql}
\DeclareMathOperator{\Char}{char}

\DeclareMathOperator{\Hom}{Hom}

\DeclareMathOperator{\rk}{rk}

\DeclareMathOperator{\modd}{mod}
\DeclareMathOperator{\supp}{supp}

\DeclareMathOperator{\Ext}{Ext}

\DeclareMathOperator{\im}{im}

\DeclareMathOperator{\GL}{GL}

\DeclareMathOperator{\El}{El}

\DeclareMathOperator{\EIP}{EIP}

\DeclareMathOperator{\Inj}{Inj}

\DeclareMathOperator{\EKP}{EKP}

\DeclareMathOperator{\dimu}{\underline{dim}}

\let\emptyset\varnothing
\newtheorem{proposition}{Proposition}[section]
\newtheorem{Theorem}[proposition]{Theorem}
\newtheorem{Lemma}[proposition]{Lemma}

\newtheorem{corollary}[proposition]{Corollary}
\newtheorem*{TheoremN}{Theorem}

\newtheorem*{corollaryN}{Corollary}
\newtheorem*{PropositionN}{Proposition}
\newenvironment{example}[1][Example.]{\begin{trivlist}
\item[\hskip \labelsep {\bfseries #1}]}{\end{trivlist}}
\newenvironment{examples}[1][Examples.]{\begin{trivlist}
\item[\hskip \labelsep {\bfseries #1}]}{\end{trivlist}}
\newenvironment{Remark}[1][Remark.]{\begin{trivlist}
\item[\hskip \labelsep {\bfseries #1}]}{\end{trivlist}}

\newenvironment{Definition}[1][Definition.]{\begin{trivlist}
\item[\hskip \labelsep {\bfseries #1}]}{\end{trivlist}}

\begin{document}

\rmfamily

%%%%%%%%%%%%%%%%%%%%%% TITELSEITE %%%%%%%%%%%%%%%%%%%%%%%%%%%%%%%%

\thispagestyle{empty}
\author{Daniel Bissinger}

\title{Dimension vectors with the equal kernels property} 
\address{Christian-Albrechts-Universit\"at zu Kiel, Ludewig-Meyn-Str. 4, 24098 Kiel, Germany}
\email{bissinger@math.uni-kiel.de}

\maketitle

\begin{abstract}
Let $r \in \NN$, $\Gamma_r$ be the generalized Kronecker quiver with $r$ arrows $\gamma_1,\ldots,\gamma_r \colon 1 \to 2$ and $\delta \in \Delta_+(\Gamma_r)$ be a positive root of $\Gamma_r$. We say that $\delta$ has the equal kernels property if for all $\alpha \in k^r \setminus \{0\}$ and every indecomposable representation $M$ with dimension vector $\dimu M = \delta$ the $k$-linear map $M^\alpha := \sum^r_{i=1} \alpha_i M(\gamma_i) \colon M_1 \to M_2$ is injective. We show that $\delta$ has the equal kernels property if and only if $q_{\Gamma_r}(\delta) + \delta_2 - \delta_1 \geq 1$, where $q_{\Gamma_r} \colon \ZZ^2 \to \ZZ, (x,y) \mapsto x^2 + y^2 - rxy$ denotes the Tits quadratic form of $\Gamma_r$.
\end{abstract}

\section{Introduction}

\noindent Let $k$ be an algebraically closed field of characteristic $p \geq 2$, $r \in \NN_{\geq 2}$ and $E_r$ be a $p$-elementary abelian group of rank $r$. 
In \cite{CFP2} and \cite{CFS1} the authors introduced and studied the subcategories of modules with the equal kernels property and the equal images property to get a better understanding of the, in general, wild category $\modd kE_r$ of finite dimensional $kE_r$-modules.
Let $\langle x_1,\ldots,x_r \rangle_k$ be a $k$-complement of $\Rad^2(kE_r)$ in $\Rad(kE_r)$ and set $x_\alpha := \sum^r_{i=1} \alpha_i x_i$ for all $\alpha \in k^r$. A module $M \in \modd kE_r$ has the \textsf{equal kernels (images) property}, provided the kernel (image) of the nilpotent operator $x^M_\alpha \colon M \to M, m \mapsto x_\alpha \cdot m$ is independent of $\alpha \in k^r \setminus \{0\}$.\\
Although the categories given by these modules appear at first glance much smaller, they have turned out to be wild categories in those cases, where $\modd kE_r$ is wild.\\
In this article we study the category $\modd_{\leq 2} kE_r$ of modules of Loewy length $\leq 2$ that is stably equivalent to the category $\rep(\Gamma_r)$ of finite dimensional representations of the generalized $r$-Kronecker quiver $\Gamma_r$. The quiver $\Gamma_r$ has two vertices $1,2$ and $r$ arrows $\gamma_1,\ldots,\gamma_r \colon 1 \to 2$. We say that $M = (M_1,M_2,\{M(\gamma_i) \colon M_1 \to M_2 \mid i \in \{1,\ldots,r\} \}) \in \rep(\Gamma_r)$ has the \textsf{equal kernels property}, provided the $k$-linear map $M^\alpha := \sum^r_{i=1} \alpha_i M(\gamma_i) \colon M_1 \to M_2$ is injective for all $\alpha \in k^r \setminus \{0\}$. \\
Since the essential image of these representations under the stable equivalence consists precisely of the modules in $\modd_{\leq 2} kE_r$ with the equal kernels property, we can use tools of the hereditary category $\rep(\Gamma_r)$ that are not available in $\modd kE_r$ to study the equal kernels property:\\
We denote by $q_{\Gamma_r} \colon \NN^2_0 \to \ZZ, (x,y) \mapsto x^2+y^2- rxy$ the \textsf{Tits form} of $\Gamma_r$ and by $\Phi_r$ the \textsf{Coxeter matrix} of $\Gamma_r$. Let $\delta \in \Delta_+(\Gamma_r)$ be a positive root of $\Gamma_r$ and $M \in \rep(\Gamma_r)$ be an indecomposable representation with dimension vector $\delta$. By Westwick's Theorem \cite{We1} we know that if $M$ has the equal kernels property, then $\delta_2 - \delta_1 \geq r - 1$ and $\delta_1 \delta_2 \neq 0$ or $\delta = (0,1)$. In general, however, the dimension vector does not give much information as to whether a representation has the equal kernels property as the following two indecomposable representations for $\Gamma_3$ show:
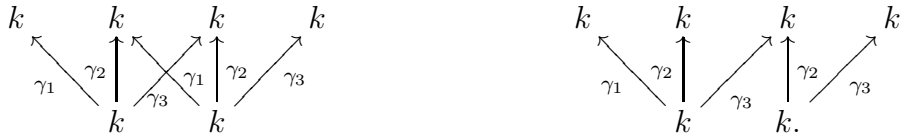
\[
\xymatrix{
k & k & k &k & & &  k & k & k &k \\  
  & k \ar^{\gamma_1}[ul] \ar^{\gamma_2}[u] \ar_>>>>>{\gamma_1}[ur]& k \ar^<<<<{\gamma_3}[ul] \ar_{\gamma_2}[u] \ar_>>>>>{\gamma_3}[ur]& &  & & & k \ar^{\gamma_1}[ul] \ar^{\gamma_2}[u] \ar_<<<<{\gamma_3}[ur]& k.  \ar_{\gamma_2}[u] \ar_{\gamma_3}[ur]
}
\]
Both representations are indecomposable with dimension vector $(2,4)$, the representation on the left hand side has the equal kernels property and the representation on the right hand side does not have the equal kernels property. \\
In this paper we study under which assumptions on $\delta \in \Delta_+(\Gamma_r)$, every indecomposable representation with dimension vector $\delta$ has the equal kernels property. We say that $\delta \in \Delta_+(\Gamma_r)$ has the \textsf{equal kernels property} if every indecomposable representation with dimension vector $\delta$ has the equal kernels property and define
\[ \underline{\EKP}(r) := \{ \delta \in \Delta_+(\Gamma_r) \mid \delta \ \ \text{has the equal kernels property} \}.\]  
Dually, we define the dimension vectors with the equal images property and denote the set of all such dimension vectors by $\underline{\EIP}(r)$.
We prove:

\begin{TheoremN}
Let $\delta \in \Delta_+(\Gamma_r)$ be a positive root. The following statements are equivalent:
\begin{enumerate}
\item [$(i)$] $\delta \in \underline{\EKP}(r) \cup \underline{\EIP}(r)$.
\item [$(ii)$] $q_{\Gamma_r}(\delta) + |\delta_1 - \delta_2| \geq 1$.
\end{enumerate}
\end{TheoremN}

The proof of $(i) \implies (ii)$ is given in Section 4. It relies on the use of covering theory, which allows us to reduce the considerations to positive roots $\delta \in \Delta_+(\Gamma_r)$ with $\delta_1 \leq \delta_2$ such that $(\delta_1,\delta_2+1)$ is no longer a positive root of $\Gamma_r$. For these roots we use a homological characterization of the representations with the equal kernels property in $\rep(\Gamma_r)$ to conclude that $q_{\Gamma_r}(\delta)+ |\delta_1 - \delta_2| \leq 0$ implies $\delta \not\in \underline{\EKP}(r) \cup \underline{\EIP}(r)$. \\
Our proof of $(ii) \implies (i)$ is inspired by \cite[4.2.2]{Wiede1}, which has been used by the author to study real root representations for certain families of quivers.

\bigskip
\noindent In Section $5$ we give applications of the proven equivalence. First we draw consequences in $\modd kE_r$ and show:

\begin{corollaryN} Let $\Char(k) = p \geq 2$, $M \in \modd kE_r$ be a $kE_r$-module and assume that $M/\Rad^2_{kE_r}(M)$ is indecomposable such that
\[q_{\Gamma_r}(m_0-m_1,m_1-m_2) + m_0 - 2m_1 + m_2 \geq 1,\] 
where $m_i := \dim_k \Rad^i_{kE_r}(M)$ for all $i \in \{0,1,2\}$. Then $M$ has the equal images property.
\end{corollaryN}

\noindent For the Coxeter orbits of imaginary root of $\Gamma_r$ the description of $\underline{\EKP}(r) \cup \underline{\EIP}(r)$ by the quadratic form $q_{\Gamma_r}$ can be stated as follows:

\begin{corollaryN}
Let $r \geq 3$ and $\cO$ be the Coxeter orbit of an imaginary root. There exist uniquely determined elements $\delta_{\cO} \in \cO$ and $m_{\cO} \in \NN_0$ such that 
\begin{enumerate}
\item[$(i)$] $\underline{\EIP}(r) \cap \cO = \{ \Phi_r^l \delta_{\cO} \mid l \in \NN\}$ and
\item [$(ii)$] $\underline{\EKP}(r) \cap \cO = \{ \Phi_r^{-l} \delta_{\cO} \mid l \geq m_\cO\}$.
\end{enumerate} 
\end{corollaryN}

\noindent As an application we get computable bounds for the invariants $\rk(\cC)$, $\cW(\cC)$, introduced in \cite{Ker1} and \cite{Wor1}, attached to a regular component $\cC$ of the Auslander-Reiten quiver of $\Gamma_r$ for $r \geq 3$.  The proof of our Theorem also reveals the following:

\begin{PropositionN}
Let $r \geq 3$ and $\cC$ be a regular component of the Auslander-Reiten quiver of $\Gamma_r$. There exists a uniquely determined quasi-simple representation $X_{\cC} \in \cC$ such that for all $N \in \cC$ the following statements are equivalent:
\begin{enumerate}
\item[$(i)$] $N$ is a successor of $X_\cC$ in $\cC$.
\item[$(ii)$] Every \textsf{elementary filtration} $($see section $\ref{Section:Elementary})$ of $N$ has only equal kernels representations as factors.
\end{enumerate}
\end{PropositionN}

\section{Preliminaries}

\subsection{Kac's Theorem}
\label{Section:Kac's Theorem}

\noindent Let $Q$ be an acyclic quiver without loops with finite vertex set $Q_0 = \{1,\ldots,n\}$. For $i \in Q_0$ we define $i^+_Q := \{ j \in Q_0 \mid \exists i \to j\}$, $i^-_Q := \{ j \in Q_0 \mid \exists j \to i \}$ and $n_Q(i) := i^+_Q \cup i^-_Q$. The quiver $Q$ defines a $($non-symmetric$)$ bilinear form $\langle \ , \ \rangle_Q \colon \ZZ^n \times \ZZ^n \to \ZZ$, given by 
\[ (x,y) \mapsto \sum^n_{i=1} x_i y_i - \sum_{i \to j \in Q_1} x_i y_j,\]
which coincides with the \textsf{Euler-Ringel form} on the Grothendieck group of $Q$, i.e. for $X,Y \in \rep(Q)$ we have 
\[ \langle \dimu X, \dimu Y \rangle_Q = \dim_k \Hom_Q(X,Y) - \dim_k \Ext^1_Q(X,Y).\]
The \textsf{Tits quadratic form} is defined by $q_Q(x) := \langle x,x \rangle_Q$. We denote the \textsf{symmetric form} corresponding to $\langle \ , \ \rangle_Q$ by $( \ , \ )_{Q}$, i.e. $(x,y)_Q := \langle x,y \rangle_Q + \langle y,x \rangle_Q$.

\noindent For each $i \in Q_0$ we have an associated reflection $r_i \colon \ZZ^n \to \ZZ^n$ given by $r_i(x) := x - (x,e_i)_Q e_i$, where $e_i \in \ZZ^n$ denotes the $i$-th canonical basis vector. By definition we have
\[ r_i(x)_j = \begin{cases}
x_j,\text{for} \ j \neq i \\ 
-x_i + \sum_{l \to i, i \to l} x_l,\text{for} \ j = i.
\end{cases}
\]
We denote by $W_Q := \langle r_i \mid i \in \{1,\ldots,n\} \rangle$ the \textsf{Weyl group} associated to $Q$, by $\Pi_Q := \{e_1,\ldots,e_n\}$  the set of \textsf{simple roots} and for $\delta \in \ZZ^n$ we define $\supp(\delta) := \{ i \in \{1,\ldots,n\} \mid \delta_i \neq 0 \}$. The set
\[F_Q := \{ \delta \in \NN^n_0 \setminus \{0\} \mid \forall i \in \{1,\ldots,n\} \colon (\delta,e_i)_Q \leq 0, \supp(\delta) \ \text{is connected}\}\] is called the \textsf{fundamental domain} of the Weyl group action.

\begin{Definition}
We define 
\[ \Delta_+(Q) = \Delta^{\re}_+(Q)\sqcup \Delta^{\im}_+(Q),\]
where $\Delta^{\re}_+(Q) := W_Q \Pi_Q \cap \NN^n_0$ and $\Delta^{\im}_+(Q) := W_Q F_Q$.
The elements in $\Delta_+(Q)$ are called \textsf{(positive) roots} of $q_Q$ or roots of $Q$.
\end{Definition}

\noindent We formulate a simplified version of Kac's Theorem that suffices for our purposes.

\begin{Theorem}[Kac's Theorem]  \label{Theorem:Kac}\cite[Theorem B]{Kac3}, \cite[Theorem \S 1.10]{Kac2}, \cite[§6]{CB1}
Let $k$ be an algebraically closed field and $Q$ an acyclic finite, connected quiver without loops and vertex set $\{1,\ldots,n\}$. Let $\delta \in \NN^n_0$. 

\begin{enumerate}
\item[$(i)$] There exists an indecomposable representation in $\rep(Q)$ with dimension vector $\delta$ if and only if $\delta \in \Delta_+(Q)$. 
\item[$(ii)$] If $\delta \in \Delta_+(Q)$, then $q_{Q}(\delta) \leq 1$.
\item[$(iii)$] If $\delta \in \Delta^{\re}_+(Q)$, then there exists a unique	indecomposable representation $M_\delta \in \rep(Q)$ with $\dimu M = \delta$.
\item[$(iv)$] If $\delta \in \Delta^{\im}_+(Q)$, there exist infinitely many, pairwise non-isomorphic indecomposable representations with dimension vector $\delta$.
\end{enumerate} 
\end{Theorem}

\subsection{Kronecker quivers}

In this section, we summarize basic facts concerning represenstations of Kronecker quivers. The reader is referred to \cite[VII]{Assem1}, \cite[XI.4]{Assem2} and \cite[XVIII]{Assem3} for more details and unexplained terminology. Let $r \geq 1$. We denote by $\Gamma_r$ the generalized Kronecker quiver with two vertices $1,2$ and $r$ arrows:

\[
\xymatrix{
\Gamma_r = & 1 \ar^{\gamma_1}_{\vdots}@/^1pc/[rr]  \ar_{\gamma_r}@/_1pc/[rr] & & 2.
}\]
The Tits form is $q_{\Gamma_r}(x_1,x_2) = x_1^2 + x_2^2 -rx_1x_2.$
Since $\Gamma_r$ is a \textsf{hyperbolic} in the sense of \cite[§1.2]{Kac2}, we also have $\Delta^{\re}_+(\Gamma_r) = \{ (a,b) \in \NN^2_0 \mid q_{\Gamma_r}(a,b) = 1 \}$ and  $\Delta^{\im}_+(\Gamma_r) = \{ (a,b) \in \NN^2 \mid q_{\Gamma_r}(a,b) \leq 0 \}$. 
Hence direct computation shows
\begin{equation}
\boxed{
\Delta^{\im}_+(\Gamma_r)  =  \begin{cases}  \emptyset, r = 1\\
 \{ (a,a) \mid a \in \NN \}, r = 2 \\ 
  \{(a,b) \in \NN^2 \mid (\frac{r-\sqrt{r^2-4}}{2}) < \frac{b}{a} <  (\frac{r+\sqrt{r^2-4}}{2})\}, r \geq 3.
  \end{cases}} \tag{$\blacktriangledown$}
\end{equation} 

For $r \geq 2$ there are infinitely many isomorphism classes of indecomposable representations of $\Gamma_r$. We denote by $\tau_{\Gamma_r}$ the Auslander-Reiten translation of $\Gamma_r$. The indecomposable representations fall into three classes: an indecomposable representation $M$ is called \textsf{preprojective} $($\textsf{preinjective}$)$ if and only if $M$ is in the $\tau_{\Gamma_r}$-orbit of a projective $($injective$)$ indecomposable representation.  All other indecomposable representations are called \textsf{regular}. We call a representation $M \in \rep(\Gamma_r)$ \textsf{preprojective} $($\textsf{preinjective}, \textsf{regular}$)$ if all indecomposable direct summands of $M$ are preprojective $($resp. preinjective, regular$)$.
There are up to isomorphism two indecomposable projective representations $P_1,P_2$, two indecomposable injective representations $I_1,I_2$ and two simple representations $I_1,P_1$.
  We define recursively $P_{i+2} := \tau^{-1}_{\Gamma_r} P_i$ and $I_{i+2} := \tau_{\Gamma_r} I_i$ for all $i \in \NN$.
 The Auslander-Reiten quiver of $\Gamma_r$  quiver looks as follows:
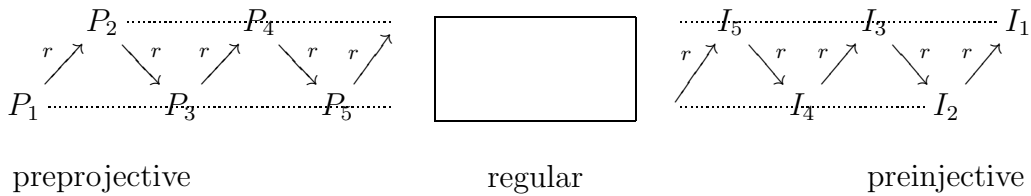
\[ 
\xymatrix @C=10pt@R=15pt{
& P_2 \ar^{r}[dr] \ar@{..}[rrrr]&  & P_4 \ar^{r}[dr] &   & &   & &  &  &  & \ar@{..}[rrrrr] & I_5 \ar^{r}[dr]  & & I_3 \ar^{r}[dr]& &  I_1\\
P_1   \ar@{..}[rrrrr] \ar^{r}[ur]&  & P_3 \ar^{r}[ur] & & P_5 \ar^{r}[ur]  & &  & & & & & \ar^{r}[ur] \ar@{..}[rrrr]  & & I_4 \ar^{r}[ur] & & I_2 \ar^{r}[ur] \\
& & & & & & & & & & & & & & & &
\save "1,7"."2,11"*[F]\frm{} \restore  
\save "3,2"."3,7" *\txt{preprojective} \restore
\save "3,9"."3,10" *\txt{regular} \restore
\save "3,16"."3,17" *\txt{preinjective} \restore}
\] 
The arrows in the preprojective and preinjective component all have multiplicity $r$. We also have for $M \in \rep(\Gamma_r)$ indecomposable the equivalence $($see \cite[2]{BoChen1}$)$
 \[ \boxed{M \ \text{regular} \Leftrightarrow q_{\Gamma_r}(\dim_k M_1,\dim_k M_2) \leq 0.} \tag{$\blacklozenge$}\]
For $r \geq 3$ every regular component $\cC$ of the Auslander-Reiten quiver is of type $\ZZ \mathbb{A}_\infty$ $($see \cite[XVIII.1.6]{Assem3}$)$. A representation $M$ in such a component is called \textsf{quasi-simple}, if the AR-sequence terminating in M has an indecomposable middle term. Let $M$ in $\cC$ be quasi-simple. There is an infinite chain $($a ray$)$ of irreducible monomorphisms $($see \cite[XVIII.1.4-1.6]{Assem3}$)$
\[ M = M[1] \to M[2] \to M[3] \to \cdots \to M[l] \to \cdots\]
\noindent in $\cC$ that is uniquely determined up to isomorphism. A representation $N$ in $\cC$ is called \textsf{successor} of $M$, provided there are $l \in \NN$ and $i \in \NN_0$ such that $N = \tau^{-i}_{\Gamma_r} M[l]$, i.e. there is an oriented path in $\cC$ from $M$ to $N$. If there is an oriented path from $N$ to $M$ in $\cC$, then $N$ is called \textsf{predecessor} of $M$.
For $X$ in $\cC$ there is a unique quasi-simple representation $N$ and $l \in \NN$ with $X = N[l]$. The number $l$ is called the \textsf{quasi-length} of $X$.

\vspace{3pt}
\noindent Direct computation shows that the Coxeter matrix $\Phi_r$ and its inverse are given by $\Phi_r = \begin{pmatrix}
r^2-1 & -r \\
r & -1 
\end{pmatrix}$ and $\Phi^{-1}_r = \begin{pmatrix}
-1 & r \\
- r & r^2 - 1
\end{pmatrix}$, respectively.

Let $M \in \rep(\Gamma_r)$ be a representation. For $\alpha \in k^r \setminus \{0\}$ we define $M^\alpha := \sum^r_{i=1} \alpha_i M(\gamma_i) \colon M_1 \to M_2$. We say that $M$ has the \textsf{equal kernels property} if  $M^\alpha$ is injective for all $\alpha \in k^r \setminus \{0\}$. The representation has the \textsf{equal images property} if $M^\alpha$ is surjective for all $\alpha \in k^r \setminus \{0\}$. We denote by $\EKP(r)$ and $\EIP(r)$ the full subcategories of $\rep(\Gamma_r)$ consisting of representations with the equal kernels property and the equal images property, respectively. The following result provides a functorial characterization of the aforementioned categories.

\begin{Theorem}\cite[2.2.1]	{Wor1}\label{Theorem:Worch}
Let $r \geq 2$. There exists a family of regular indecomposable representations $(X_\alpha)_{\alpha \in k^r \setminus \{0\}}$, such that the following statements hold:
\begin{enumerate}
\item[$(1)$] $\EKP(r) = \{ M \in \rep(\Gamma_r) \mid \forall \alpha \in k^r \setminus \{0\} \colon \Hom_{\Gamma_r}(X_\alpha,M) = 0 \}$.
\item[$(2)$] $\EIP(r) = \{ M \in \rep(\Gamma_r) \mid \forall \alpha \in k^r \setminus \{0\} \colon \Ext^1_{\Gamma_r}(X_\alpha,M) = 0 \}$.
\item[$(3)$] $\{ X_\alpha \mid \alpha \in k^r \setminus \{0\}\} = \{ X \in \rep(\Gamma_r) \mid X \ \text{indecomposable}, \ \dimu X = (1,r-1)\}$.

\end{enumerate}
\end{Theorem}

\subsection{The universal covering} \label{Subsection:UniversalCovering}

In this section we assume that $r \geq 2$. We consider the universal cover $C_r$ of the quiver $\Gamma_r$. The quiver $C_r$ is an $($infinite$)$ $r$-regular tree with bipartite orientation. We let $C^+_r$ be the set of all sources of $C_r$, $C_r^-$ be the set of all sinks and denote by $\rep(C_r)$ the category of finite dimensional representations of $C_r$.  We only recall those properties that are relevant for our purposes. For a more detailed description we refer to \cite{Gab3},\cite{Ri7} and \cite{Bi2}.\\ We fix a covering morphism $\pi \colon C_r \to\Gamma_r$ of quivers, i.e. $\pi$ is a morphism of quivers and for each $x \in (C_r)_0$ the induced map $n_{C_r}(x) \to n_{\Gamma_r}(\pi(x))$ is bijective.

\noindent By \cite[3.2]{Gab2} there exists an exact functor $\pi_\lambda \colon \rep(C_r) \to \rep(\Gamma_r)$ such that $\pi_{\lambda}(M)_{1} = \bigoplus_{x \in C^+_r} M_x$, $\pi_{\lambda}(M)_{2} = \bigoplus_{y \in C_r^-} M_y$ and $\pi_\lambda(M)(\gamma_i) = \bigoplus_{\gamma \in \pi^{-1}(\gamma_i)} M(\gamma)$ for all $i \in \{1,\ldots,r\}$. Morphisms are defined in the obvious way. 

\begin{Theorem}\cite[3.6]{Gab3}, \cite[6.2, 6.3]{Ri7}\label{TheoremRingelGabriel}
The following statements hold:
\begin{enumerate}[topsep=0em, itemsep= -0em, parsep = 0 em, label=$(\alph*)$]
\item $\pi_{\lambda}$ sends indecomposable representations in $\rep(C_r)$ to indecomposable representations in $\rep(\Gamma_r)$.
\item The free group $G(r)$ of rank $r-1$ acts on $\rep(\Gamma_r)$ via auto-equivalences such that if $M \in \rep(C_r)$ is indecomposable, then $\pi_{\lambda}(M) \cong \pi_{\lambda}(N)$ if and only if $M^g \cong N$ for some $g \in G(r)$.
\item The category $\rep(C_r)$ has almost split sequences, $\pi_{\lambda}$ sends almost split sequences to almost split sequences and $\pi_{\lambda}$ commutes with the Auslander-Reiten translates, i.e. 
\[ \tau_{\Gamma_r} \circ \pi_{\lambda} = \pi_{\lambda} \circ \tau_{C_r}.\]
\end{enumerate}
\end{Theorem}

\noindent The next result tells us that it is not hard to decide whether the push-down $\pi_\lambda(M)$ of a representation $M \in \rep(C_r)$ has the equal kernels property.

\begin{Theorem}\label{Theorem:INJEKP}\cite[4.1]{Bi2}
\label{PropositionGeneral}
Let $M \in \rep(C_r)$ be an indecomposable representation. The following statements are equivalent:
\begin{enumerate}[topsep=0em, itemsep= -0em, parsep = 0 em, label=$(\alph*)$]
\item $N := \pi_{\lambda}(M) \in \EKP(r)$.
\item $N(\gamma_i)$ is injective for all $i \in \{1,\ldots,r\}$.
\item $M \in \Inj(r) := \{ L \in \rep(C_r) \mid \forall \gamma \in (C_r)_1: L(\gamma) \ \text{is injective}\}$.
\end{enumerate}
\end{Theorem} 

\begin{Lemma}\cite[5.2.1]{Bi2}\label{Lemma:DimensionShift}
Let $M$ be in $\rep(C_r)$ indecomposable and not injective. 
\begin{enumerate}[topsep=0em, itemsep= -0em, parsep = 0 em, label=$(\alph*)$]  
\item For each $x \in C_r^+$ we have $\dim_k (\tau^{-1}_{C_r} M)_x = (\sum_{y \in x^+_{C_r}} \dim_k M_y) - \dim_k M_x.$
\item For each $y \in C_r^-$ we have $\dim_k (\tau^{-1}_{C_r} M)_y = (\sum_{x \in y^-_{C_r}} \dim_k (\tau^{-1}_{C_r} M)_x) - \dim_k M_y.$
\end{enumerate}
\end{Lemma}

\noindent In the rest of this section we make the preparations needed for the proof of Theorem $\ref{Theorem:NotMaximalRank1}$ in section $\ref{Section:4}$. Recall that for $M \in \rep(C_r)$ the \textsf{support of $M$}  is defined as $\supp(M) := \{ x \in (C_r)_0 \mid M_x \neq 0 \}$. If $M$ is indecomposable then $\supp(M)$ is a finite tree and $x \in \supp(M)$, then $x$ is called a \textsf{leaf of $M$}, provided $|\supp(M) \cap n_{C_r}(x)| \leq 1$.

\begin{Definition}
Let $M$ in $\rep(C_r)$ be indecomposable. We say that $M$ has a \textsf{thin sink branch} if there exist $x,y \in \supp(M)$ such that 
\begin{enumerate}
\item $x \in n_{C_r}(y)$ and $x$ is a source,
\item $x$ and $y$ are $M$-\textsf{thin}, i.e. $\dim_k M_x = 1 = \dim_k M_y$, and
\item $y$ is a leaf of $M$.
\end{enumerate}
\end{Definition}

\begin{corollary}\label{Corollary:DimensionShift}
Let $M$ in $\rep(C_r)$ be indecomposable with thin sink branch, then $\tau^{-l}_{C_r} M$ has a thin sink branch for every $l \geq 0$.
\end{corollary}
\begin{proof} Clearly, it sufficies to prove the statement for $l = 1$.
Let $x \to y$ be a thin sink branch with leaf $y$. Consider a path $x \to y \leftarrow a \to b$ with $x \neq a$ and $y \neq b$. We are in the situation of Figure $\ref{Figure:Leaves}$.
\begin{figure}[h]
\centering 
\tikzstyle{every node}=[]
\tikzstyle{edge from child}=[]

\begin{tikzpicture}[->,>=stealth',auto,node distance=3cm,
  thick,main node/.style={circle,draw,font=\sffamily\Large\bfseries}]

  \draw [fill =gray!10] plot [smooth cycle] coordinates {(1.2,0.4) (1.2,-0.4) (-3.5,-0.8) (-3.5,0.8)};

  \node (1) at (-0.5,0) {$\circ$};
  \node (2) at (1,0) {$\circ$};
  \node (3) at (2,-1) {$\circ$};
  \node (4) at (3,0) {$\circ$};
  
    \node at (-0.5,-0.3) {$x$};
    \node at (1,-0.3) {$y$}; 
    \node at (2,-1.3) {$a$}; 
    \node at (3,-0.3) {$b$}; 
     \node at (-2.5,-0.0) {$\supp(M)$}; 
  \path[every node/.style={font=\sffamily\small}]
    
    (1) edge[] node [left] {} (2)
    (3) edge[] node [left] {} (2)
    (3) edge[] node [left] {} (4)
   ;
\end{tikzpicture}
\caption{Illustration of $\supp(M)$ and the oriented path.}
\label{Figure:Leaves}
\end{figure}
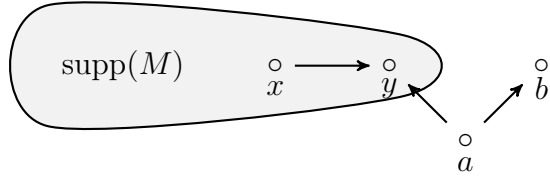
Lemma $\ref{Lemma:DimensionShift}$ implies $\dim_k (\tau^{-1}_{C_r} M)_a = \dim_k M_y = 1$. Let $z \in n_{C_r}(b) - \{a\}$, we get with $\ref{Lemma:DimensionShift}$ that $(\tau^{-1}_{C_r} M)_z = 0$. We conclude 
\[ \dim_k (\tau^{-1}_{C_r} M)_b = (\sum_{u \in n_{C_r}(b)} \dim_k (\tau^{-1}_{C_r} M)_u) - \dim_k M_b = \dim_k (\tau^{-1}_{C_r} M)_a - 0 = 1.\]
This shows that $a \to b$ is a thin sink branch of $\tau^{-1}_{C_r} M$.
\end{proof}

\noindent The proofs of the three results below may be found in \cite[3.10, 3.9 and 3.11]{Bi3}, altough they are stated slighty different. Recall that a representation $M \in \rep(C_r)$ is called \textsf{thin}, provided $\dim_k M_x \leq 1$ for all $x \in (C_r)_0$

\begin{Lemma}\label{Lemma:ExistenceThinInjectiveArrow}
Let $(u,v) \in \NN^2$ such that $u \leq v \leq (r-1)u + 1$. There exists an indecomposable and thin representation $S_{(u,v)} \in \rep(C_r)$ such that $\dimu \pi_\lambda(S_{(u,v)}) = (u,v)$ and  $S_{(u,v)}(\gamma)$ is injective for all $\gamma \in (C_r)_1$ such that $\pi(\gamma) = \gamma_1$.
\end{Lemma}

\begin{proposition}\label{Proposition:Result2}
Let $(u,v) \in \NN^2$ such that $(r-1)u+1 \leq v \leq (r-\frac{1}{r-1})u$. There exists an indecomposable representation $T_{(u,v)} \in \rep(C_r)$ such that $T_{(u,v)} \in \Inj(r)$, $\dimu \pi_\lambda (T_{(u,v)}) = (u,v)$ and each leaf of $T_{(u,v)}$ is thin. 
\end{proposition}

\begin{proposition}\label{Proposition:Result3}
Let $(u,v) \in \NN^2$ such that $(r - \frac{1}{r-1})u < v < u(\frac{r+\sqrt{r^2-4}}{2})$. Then there exists $l \in \NN$ such that for $(x,y) := \Phi_r^l(u,v)$  one of the following cases holds:
\begin{enumerate}
\item $x \leq y \leq (r-1)x + 1$ and $\ref{Lemma:ExistenceThinInjectiveArrow}$ yields an indecomposable representation $S_{(x,y)}$ such that $\tau^{-l}_{C_r} S_{(x,y)}$ satisfies $\dimu \pi_{\lambda}(\tau^{-l}_{C_r} S_{(x,y)}) = (u,v)$. 
\item $(r-1)x+1 \leq y \leq (r-\frac{1}{r-1})y$ and $\ref{Proposition:Result2}$ yields an indecomposable representation $T_{(x,y)}$ such that $\tau^{-l}_{C_r} T_{(x,y)}$ satisfies $\dimu \pi_{\lambda}(\tau^{-l}_{C_r} S_{(x,y)}) = (u,v)$. 
\end{enumerate} 
\end{proposition}

\begin{corollary}\label{Corollary:ThinSinkBranch}
Let $(a,b) \in \Delta^{\im}_+(\Gamma_r)$ be an imaginary root such that $a \leq b$. There exists an indecomposable representation $B_{(a,b)} \in \rep(C_r)$ such that $B_{(a,b)}$ has a thin sink branch and $\dimu \pi_\lambda(B_{(a,b)})= (a,b)$
\end{corollary}
\begin{proof}
We first consider the case that $a \leq b \leq (r-1)a +1$. Let $S_{(a,b)}$ be a representation as in Lemma $\ref{Lemma:ExistenceThinInjectiveArrow}$. Since $(a,b)$ is an imaginary root, $a \neq 0$ and we find a source $x \in \supp(S_{(a,b)})$. Let $\gamma \colon x \to y$ be the unique arrow starting in $x$ such that $\pi(\gamma) = \gamma_1$. Since $S_{(a,b)}(\gamma)$ is injective, we conclude $y \in \supp(S_{(a,b)})$. If $y$ is a leaf, we are done since $S_{(a,b)}$ is thin. Otherwise we find $z \in n_{C_r}(y) - \{x\} \cap \supp(S_{(a,b)})$. Since the underlying graph of $\supp(S_{(a,b)})$ is a finite tree, we can continue the argument until we find a leaf of $S_{(a,b)}$ in $C^-_r$, say $q$. Let $p$ be the unique vertex in $n_{C_r}(q) \cap \supp(S_{(a,b)})$. Then $p \to q$ is a thin sink branch.\\
For $(r-1)a+1 \leq b \leq (r-\frac{1}{r-1})a$ we consider $T_{(a,b)}$ as in Proposition $\ref{Proposition:Result2}$ and note that each leaf of $T_{(a,b)}$ is in $C^-_r$. Let $x \to y$ be a leaf branch. Since $T_{(a,b)}(x \to y)$ is injective and $\dim_k {T_{(a,b)}}_y = 1$, $T_{(a,b)}$ has a thin sink branch.\\
For $(r - \frac{1}{r-1})a < b < a(\frac{r+\sqrt{r^2-4}}{2})$ we apply the above considerations in conjunction with Proposition $\ref{Proposition:Result3}$ and Corollary $\ref{Corollary:DimensionShift}$.
 \end{proof}

\begin{proposition}\label{Proposition:ExistenceThinSinkBranch}
Let $(a,b) \in \Delta_+(\Gamma_r)$ be a root of $\Gamma_r$ such that $1 \leq a \leq b$. Then there is an indecomposable representation $B_{(a,b)} \in \rep(C_r)$ with thin sink branch such that $\dimu \pi_\lambda(B_{(a,b)}) = (a,b)$.
\end{proposition}
\begin{proof}
In view of Corollary $\ref{Corollary:ThinSinkBranch}$ it remains to consider the case $q_{\Gamma_r}(a,b) = 1$. We conclude with $(\blacklozenge)$ and $1 \leq a \leq b$ that $(a,b)$ is the dimension vector of an indecomposable preprojective representation that is not simple. Hence we find $i \geq 2$ such that $(a,b) = \dimu P_i$.\\
We fix a source $x \in (C_r)_0$ and consider the thin indecomposable representation $X \in \rep(C_r)$ such that $\supp(X) = \{x\} \cup n_{C_r}(x)$.  Let $y \in n_{C_r}(x)$ and consider the simple representation $Y$ in $y$. Application of Lemma $\ref{Lemma:DimensionShift}$ shows that $\tau^{-1}_{C_r} Y$ has a thin sink branch. We conclude with Corollary $\ref{Corollary:DimensionShift}$ that $\tau^{-l}_{C_r}( \tau^{-1}_{C_r} Y) = \tau^{-(l+1)}_{C_r} Y$ and $\tau^{-l}_{C_r} X$ have a thin sink branch for all $l \geq 0$. \\
Note that $\dimu \pi_\lambda(X) = (1,r) = \dimu P_2$ and  $\dimu \pi_\lambda(\tau^{-1}_{C_r} Y) = (r,r^2-1) = \dimu P_3$. Since $\pi_\lambda$ interchanges with $\tau^{-l}_{C_r}$ for all $l \geq 0$, we conclude that $\dimu \pi_\lambda(\tau^{-l}_{C_r} X) = \dimu P_{2l+2}$ and  $\dimu \pi_\lambda(\tau^{-(l+1)}_{C_r} Y) = \dimu P_{2l+3}$. 
\end{proof}

\section{An inequality given by the quadratic form}

Let $Q = (Q_0,Q_1,s,t)$ be a finite quiver. For $x \in Q_0$ we define $Q^{x,s}_1 := \{ \gamma \in Q_1 \mid s(\gamma) = x\}$ as well as $Q^{x,t}_1 := \{ \gamma \in Q_1 \mid t(\gamma) = x\}$. In \cite{Wiede2} the author introduced the notion of representations of maximal rank type for representations of $Q$. A representation $M \in \rep(Q)$ has \textsf{maximal rank type} if for each vertex $x \in Q_0$ and all non-empty subsets $\cA \subseteq Q^{x,s}_1$, $\cB \subseteq Q^{x,t}_1$ the natural $k$-linear maps 
\[ M_{x,s,\cA} := \bigoplus_{\gamma \in \cA} M_{s(\gamma)} \to M_x, \quad M_{x,t,\cB} := M_x \to \bigoplus_{\gamma \in \cB} M_{t(\gamma)}\]
have maximal rank. In his thesis \cite{Wiede1} he gave a refined version of this definition, that allowed arbitrary non-trivial linear combinations of the involved maps, and he proved that if $\delta \in \Delta^{\re}_+(Q)$ is a real root of $Q$, then the unique indecomposable representation with dimension vector $\delta$ has maximal rank type. We adapt his nice proof of this result to our situation to show:

\begin{Theorem}$($see \cite[4.2.2]{Wiede1}$)$\label{Proposition:Wiedemann}
Let $\delta \in \Delta_+(\Gamma_r)$ be such that $q_{\Gamma_r}(\delta) + |\delta_1 - \delta_2| \geq 1$. Let $M \in \rep(\Gamma_r)$ be indecomposable such that $\dimu M = \delta$, then $M$ has the equal kernels property or the equal images property.
 \end{Theorem}
\begin{proof}
The duality $D_{\Gamma_r} \colon \rep(\Gamma_r) \to \rep(\Gamma_r)$ introduced in \cite[2.2]{Wor1} satisfies $D_{\Gamma_r}(\EKP(r)) = \EIP(r)$ and $\dim_k (D_{\Gamma_r}L)_i = \dim_k L_{3-i}$ for all $L \in \rep(\Gamma_r)$ and $i \in \{1,2\}$. Therefore we can assume without loss of generality that $\delta_1 \leq \delta_2$. Let $\alpha \in k^r \setminus \{0\}$. We show that $M^\alpha \colon M_1 \to M_2$ is injective. Let $(\alpha,\beta_2,\ldots,\beta_r)$ be a basis of $k^r$. We define a new representation $X$ for the quiver $\widehat{\Gamma}_r$
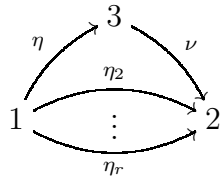
\[
\xymatrix{
 & 3 \ar^{\nu}@/^/[rd]&  \\
 1 \ar^{\eta_2}_{\vdots}@/^1pc/[rr]  \ar_{\eta_r}@/_1pc/[rr] \ar^{\eta}@/^/[ru]& & 2 
}\]
by the following data: $X_i := M_i$ for $i \in \{1,2\}$, $X_3 := \im M^\alpha$, $X(\eta) := M^\alpha$, $X(\nu) \colon \im M^\alpha \to M_2$ is the natural embedding and $X(\eta_i) := M^{\beta_i}$ for $i \in \{2,\ldots,r\}$. We claim that $X \in \rep(\widehat{\Gamma}_r)$ is indecomposable:\\
We define the representation $N \in \rep(\Gamma_r)$ by setting $N_j := M_j$ for $j \in \{1,2\}$, $N(\gamma_1) := M^\alpha$ and $N(\gamma_i) = M^{\beta_i}$ for all $i \in \{1,\ldots,r\}$. 
Given $Y \in \rep(\widehat{\Gamma}_r)$ we assign the representation $\cF(Y) \in \rep(\Gamma_r)$ given by $\cF(Y)_i := Y_i$ for $i \in \{1,2\}$, $\cF(Y)(\gamma_1) := Y(\nu) \circ Y(\eta)$ and $\cF(Y)(\gamma_i) := Y(\eta_i)$ for $i \in \{2,\ldots,r\}$. Cleary $\cF(U \oplus V) \cong U \oplus V$ for all $U,V \in \rep(\widehat{\Gamma}_r)$.\\
Note that $\GL_r(k)$ acts on $\rep(\Gamma_r)$ via base change, i.e. $(g^{-1}.L)(\gamma_j) = \sum^r_{i=1} g_{ij}L(\gamma_i)$ for all $L \in \rep(\Gamma_r)$, $g \in \GL_r(k)$ and $j \in \{1,\ldots,r\}$. Since $N \cong g^{-1}.M$ for the element $g \in \GL_r(k)$ with columns $\alpha,\beta_2,\ldots,\beta_r$, we conclude that $N$ is indecomposable.\\
Assume that $X = U \oplus V$. Then $\cF(U) \oplus \cF(V) \cong \cF(X) \cong N$ is indecomposable and without loss of generality $\cF(U) = 0$.  Hence $U_1 = 0 = U_2$. Since $U$ is a subrepresentation of $X$ and $X(\nu) \colon X_3 \to X_1$ is injective, we conclude that $U(\nu) \colon U_3 \to U_2$ is injective. Hence $U = 0$ and $X$ is indecomposable.

\vspace{3pt}

\noindent We conclude with Theorem $\ref{Theorem:Kac}(i),(ii)$ that
\begin{align*}
1 \geq q_{\widehat{\Gamma}_r}(\dimu X) &= (\dim_k M_1)^2 + (\dim_k M_2)^2 + (\dim_k \im M^\alpha)^2  - (r-1)(\dim_k M_1)(\dim_k M_2) \\
&\ \ \  - (\dim_k M_1)(\dim_k \im M^\alpha) - (\dim_k \im M^\alpha)(\dim_k M_2)\\
&= q_{\Gamma_r}(\dimu M) + (\dim_k \im M^\alpha)^2 \\
&\ \ \  +(\dim_k M_1)(\dim_k M_2) - (\dim_k M_1)(\dim_k \im M^\alpha) - (\dim_k M_2)(\dim_k \im M^\alpha)\\
&=q_{\Gamma_r}(\delta) + (\dim_k M_1 - \dim_k \im M^\alpha)(\dim_k M_2 - \dim_k \im M^\alpha).
\end{align*}
If $\dim_k \im M^\alpha \neq \dim_k M_1$, then $1-q_{\Gamma_r}(\delta) \geq \dim_k M_2 - \dim_k \im M^\alpha > \dim_k M_2 - \dim_k M_1$. Hence $0 \geq  q_{\Gamma_r}(\delta) + \dim_k M_2 - \dim_k M_1 = q_{\Gamma_r}(\delta) + |\delta_1 - \delta_2|$, a contradiction. This shows that $M^\alpha$ is injective. Since $\alpha \in k^r - \{0\}$ was arbitrary, we conclude $M \in \EKP(r)$.
\end{proof}

\begin{Definition}
Let $\delta \in \Delta_+(\Gamma_r)$ be a positive root of $\Gamma_r$. We say that $\delta$ has the \textsf{equal kernels property}, provided $M \in \EKP(r)$ for every indecomposable representation $M \in \rep(\Gamma_r)$ with $\dimu M = \delta$. We say that $\delta$ has the \textsf{equal images property}, provided every indecomposable representation $M \in \rep(\Gamma_r)$ with dimension vector $\dimu M = \delta$ is located in $\EIP(r)$. We denote the corresponding sets of dimension vectors by $\underline{\EKP}(r)$ and $\underline{\EIP}(r)$, respectively.
\end{Definition}

\begin{corollary}\label{Corollary:Wiedemann} The following statements hold.
\begin{enumerate}
\item We have $\Delta^{\re}_+(\Gamma_r) \subseteq \underline{\EKP}(r) \cup \underline{\EIP}(r)$.
\item We have $ \underline{\EKP}(1) \cup \underline{\EIP}(1) = \Delta_+(\Gamma_1) = \{\delta \in \Delta_+(\Gamma_1) \mid q_{\Gamma_1}(\delta)+|\delta_1-\delta_2| \geq 1 \}$.
\item We have $\underline{\EKP}(2) \cup \underline{\EIP}(2) = \Delta^{\re}_+(\Gamma_2) = \{\delta \in \Delta_+(\Gamma_2) \mid q_{\Gamma_2}(\delta)+|\delta_1-\delta_2| \geq 1 \}$.
\item For $r \geq 2$ we have $\underline{\EIP}(r) \cap \underline{\EKP}(r) = \emptyset$.
\end{enumerate}
\end{corollary}
\begin{proof} $(a)$ This follows immediatly from $\ref{Proposition:Wiedemann}$ since $\delta \in \Delta^{\re}_+(\Gamma_r)$ implies that $q_{\Gamma_r}(\delta) = 1$. \\
$(b)$ Recall that $q_{\Gamma_1}$ is positive definite and apply $(a)$.\\
$(c)$ Let $\delta \in \Delta^{\im}_+(\Gamma_2)$. Then $(\delta_1 - \delta_2)^2 = q_{\Gamma_2}(\delta) \leq 0$ and therefore $\delta_1 = \delta_2$. By Westwick's Theorem every indecomposable representation with dimension vector $\delta$ is not in $\EKP(r) \cup \EIP(r)$. Hence $\delta \not\in \underline{\EKP}(r) \cup \underline{\EIP}(r)$. We also have in this case $q_{\Gamma_2}(\delta) + |\delta_1 - \delta_2| = 0 \leq 1$. This proves $(c)$.\\
$(d)$  This follows for example from Westwick's Theorem.
\end{proof}

\section{Dimension vectors with the equal kernels property}\label{Section:4}

\noindent From now on we assume that $r \geq 3$, let $L_r := \frac{r+\sqrt{r^2-4}}{2}$ and denote by $\lfloor \ \rfloor \colon \mathbb{R}_{\geq 0} \to \NN_0$ the floor function. Recall from $(\blacktriangledown)$ that if $(a,b) \in \NN^2$ and $a \leq b$, then $(a,b)$ is an imaginary root if and only if $a \leq b \leq a L_r$.  Note that 
\begin{equation}
\boxed{r - 1 < L_r < r \ \text{and} \  r - L_r = \frac{1}{L_r}.} \tag{$\star$}
\end{equation}

\noindent In this section we show that each positive root $\delta$ in $\underline{\EKP}(r) \cup \underline{\EIP}(r)$ satisfies $q_{\Gamma_r}(\delta) + |\delta_1 - \delta_2| \geq 1$.

\subsection{Restrictions on the imaginary root $\delta$}

In view of the preliminary considerations from section $\ref{Subsection:UniversalCovering}$, we already get the following restriction for dimension vectors in $\underline{\EKP}(r) \cup \underline{\EIP}(r)$:

\begin{Theorem}\label{Theorem:NotMaximalRank1}
Let $(a,b) \in \Delta^{\im}_+(\Gamma_r)$ be an imaginary root such that $a \leq b <  \lfloor a L_r \rfloor$ or $b \leq a < \lfloor b L_r \rfloor$, then $(a,b) \notin \underline{\EKP}(r) \cup \underline{\EIP}(r)$.
\end{Theorem}
\begin{proof}
We consider the case $a \leq b < \lfloor a  L_r \rfloor$. We set $u := a$ and $v: = b + 1$. Since $b < \lfloor a L_r \rfloor$, we have $u \leq v \leq \lfloor u  L_r \rfloor$. In view of Proposition $\ref{Proposition:ExistenceThinSinkBranch}$ we find an indecomposable representation $B_{(u,v)} \in \rep(C_r)$ such that $\dimu \pi_\lambda(B_{(u,v)}) = (u,v)$ and $B_{(u,v)}$ has a thin sink branch $x \to y$. Let $Q \subseteq C_r$ be the full connected subquiver with vertex set $\supp(B_{(u,v)})$. We know that $\beta := \dimu B_{(u,v)}$ is a root for $Q$. We apply the reflection at $y$ and conclude that the vector $\beta^\prime$ given by
\[\beta^\prime_z = \begin{cases}
\beta_z, z \in Q_0 - \{y\} \\
 0, z = y
\end{cases}
\] is a root for $Q$. Hence we find an indecomposable representation $F \in \rep(Q) \subseteq \rep(C_r)$ such that $\dimu F = \beta^\prime$. In particular, $\dim_k F_x = 1$, $\dim_k F_y = 0$ and $\dimu \pi_\lambda(F) = (u,v-1) = (a,b+1-1) = (a,b)$. We conclude with $\ref{Theorem:INJEKP}$ that $\pi_\lambda(F) \not\in \EKP(r)$, since the map $F(x \to y)$ is not injective. Hence we have found an indecomposable representation with dimension vector $(a,b)$ that is not in $\EKP(r)$.
\end{proof}

\begin{corollary}\label{Corollary:DimensionVectorIsMaximal}
Let $(a,b) \in \Delta^{\im}_+(\Gamma_r)$ such that $q_{\Gamma_r}(a,b) + |a - b| \geq 1$, then $b = \lfloor a  L_r \rfloor$ or $a = \lfloor b  L_r \rfloor$.
\end{corollary}
\begin{proof}
This now follows from Theorem $\ref{Proposition:Wiedemann}$ and Theorem $\ref{Theorem:NotMaximalRank1}$.
\end{proof}

\subsection{Filtrations with regular filtration factors}

\noindent Let $(a,b) \in \Delta^{\im}_+(\Gamma_r)$ be an imaginary root and assume that $(a,b) \in \underline{\EKP}(r) \cup \underline{\EIP}(r)$. By duality, we can assume that $a \leq b$. We would like to show that $q_{\Gamma_r}(a,b) + b-a = q_{\Gamma_r}(a,b) + |a-b| \geq 1$. In view of Theorem $\ref{Theorem:NotMaximalRank1}$ we can therefore assume that $1 \leq a \leq b = \lfloor a L_r \rfloor$. \\
\noindent In the following we prove that if an indecomposable representation $M$ with dimension vector $(a,b)$ and $1 \leq a \leq b = \lfloor a L_r \rfloor$ has a filtration 
\[ 0 = M_0 \subset M_1 \subset \ldots \subset M_{n-1} \subset M_n = M\]
with regular indecomposable filtration factors $M_i/M_{i-1}$, then there is at most one $l \in \{1,\ldots,n\}$ such that $M_l/M_{l-1}$ does not have the equal kernels property. Since some of the proofs are rather technical, we have relegated them to the appendix. 
\noindent Recall that we assume throughout that  $r \geq 3$.

\begin{Lemma}\label{Lemma:DistanceToCa}
Let $1 \leq a \leq b \leq a L_r$ such that $q_{\Gamma_r}(a,b)+b-a \leq 0$, then 
\[ a L_r - b \geq \frac{1}{2}.\]
\end{Lemma}

\begin{Remark}
\noindent Note that the assumption $q_{\Gamma_r}(a,b)+b-a \leq 0$ is necessary, since for $r = 3$, $(a,b) = (2,5) \in \Delta^{\im}_+(\Gamma_3)$ we have $a L_r - b \cong 0.237 < \frac{1}{2}$.
\end{Remark}

\begin{Lemma}\label{Lemma:Average}
Let $1 \leq a \leq b = \lfloor a L_r \rfloor$ and assume there are $n \in \NN$ and $a_i,b_i \in \NN$ such that $\sum^n_{i=1} a_i = a, \sum^n_{i=1} b_i = b$ with $q_{\Gamma_r}(a_i,b_i) \leq 0$ for all $i \in \{1,\ldots,n\}$. Then $a_i\leq b_i = \lfloor a_i L_r \rfloor$ for all $i \in \{1,\ldots,n\}$.
\end{Lemma}

\begin{proposition}\label{Proposition:Filtration}
Let $M \in \rep(\Gamma_r)$ be a regular indecomposable representation such that $\dimu M = (a,b)$ satisfies $a \leq b = \lfloor aL_r \rfloor$. Let 
\[0 = M_0 \subset M_1 \subset \ldots \subset M_{n-1} \subset M_n = M\]
be a filtration such that $M_i/M_{i-1}$ is regular indecomposable for all $i \in \{1,\ldots,n\}$. We define $(a_i,b_i) := \dimu M_i/M_{i-1}$ for all $i \in \{1,\ldots,n\}$. Then $a_i \leq b_i = \lfloor a_i L_r \rfloor$ for all $i \in \{1,\ldots,n\}$ and one of the following statements holds.
\begin{enumerate}
\item[$(i)$] For all $i \in \{1,\ldots,n\}$ we have $q_{\Gamma_r}(a_i,b_i) + b_i - a_i \geq 1$ and $M \in \EKP(r)$.
\item [$(ii)$] There exists exactly one $l \in \{1,\ldots,n\}$ such that $q_{\Gamma_r}(a_l,b_l) + b_l - a_l < 1$.
\end{enumerate}
\end{proposition}
\begin{proof}
Cleary we can assume that $n \geq 2$. Note that 
\[(a,b) = \dimu M = \sum^n_{i=1} \dimu M_i/M_{i-1} = \sum^n_{i=1} (a_i,b_i).\]
Since $M_i/M_{i-1}$ is regular indecomposable, we conclude with $(\blacklozenge)$ that $q_{\Gamma_r}(a_i,b_i) \leq 0$ for all $i \in \{1,\ldots,n\}$ and therefore $1 \leq a_i \leq b_i = \lfloor a_i L_r \rfloor$ by Lemma $\ref{Lemma:Average}$. We assume that there are $i \neq j \in \{1,\ldots,n\}$ such that $q_{\Gamma_r}(a_i,b_i) + b_i - a_i \leq 0$ and $q_{\Gamma_r}(a_j,b_j) + b_j - a_j \leq 0$.
Then Lemma $\ref{Lemma:DistanceToCa}$ implies that $a_i + a_j \leq b_i+b_j \leq L_r(a_i+a_j)-1$. Hence $b_i+b_j < \lfloor L_r(a_i+a_j) \rfloor$, in contradiction to Lemma $\ref{Lemma:Average}$, since 
 \[(a,b) = (a_i+a_j,b_i+b_j) + \sum_{l \in \{1,\ldots,n\} \setminus \{i,j\}} (a_l,b_l)\]
and $q_{\Gamma_r}(a_i+a_j,b_i+b_j) \leq 0$ by $(\blacktriangledown)$. Hence there is at most one element $l \in \{1,\ldots,n\}$ such that $q_{\Gamma_r}(a_l,b_l) + b_l - a_l \leq 0$. If there is no such $l$, then Theorem $\ref{Proposition:Wiedemann}$ implies that $M_i/M_{i-1} \in \EKP(r)$ for all $i \in \{1,\ldots,n\}$ and therefore $M \in \EKP(r)$, as $\EKP(r)$ is closed under extensions.
\end{proof}

\subsection{Proof of the main theorem}

\begin{Lemma}\label{Lemma:NonSplit}
Let $\Lambda$ be a finite dimensional $k$-algebra, $X,Y$ be finitely generated $\Lambda$-modules and $Y$ be indecomposable. Let $\delta \colon 0 \to X \stackrel{f}{\to} E \to Y \to 0$ be a short exact sequence.  Assume that $U \mid E$ is a non-zero direct summand of $E$ such that $\Hom_\Lambda(X,U) = 0$, then $\delta$ splits.
\end{Lemma}
\begin{proof}
Clearly $U \neq E$ since $f \in \Hom_\Lambda(X,E)$ is non-zero. We write $U \oplus V = E$. Since $\Hom_{\Lambda}(X,U) = 0$, we have $\im f \subseteq V$ as therefore $Y \cong E/\im f \cong U \oplus (V/\im f)$. Since $Y$ is indecomposable and $U \neq 0$, we conclude $U \cong Y$ and $X \cong V$. Hence $E \cong X \oplus Y$ and $\delta$ splits, since every module is finite dimensional over $k$.
\end{proof}

\noindent The proof of the following result may be found in the appendix.

\begin{proposition}\label{Proposition:Technical}
Let $1 < a \leq b = \lfloor a L_r \rfloor$ such that $q_{\Gamma_r}(a,b) + b -a \leq 0$, then \begin{enumerate}
\item [$(i)$] $a -1 \leq b - (r-1)$,
\item[$(ii)$] $(a-1,b-(r-1))$ and $(x,y) := \Phi_r(a-1,b-(r-1))$ are imaginary roots, and
\item[$(iii)$] $q_{\Gamma_r}(x,y) + y - x \leq 0$.
\end{enumerate} 
\end{proposition}

\begin{corollary}\label{Corollary:ExistenceNonSplit}
Let $(a,b) \in \Delta^{\im}_+(\Gamma_r)$ be an imaginary root such that $1 < a \leq b = \lfloor a L_r \rfloor$ and $q_{\Gamma_r}(a,b)+b-a \leq 0$. Let $M$ be
an indecomposable representation with dimension vector $(a-1,b-(r-1))$ and $\alpha \in k^r -0$. Then every non-split short exact sequence 
\[\delta \colon 0 \to X_\alpha \to E \to M \to 0\]
has an indecomposable middle term, where $X_\alpha$ belongs to the family of indecomposable representations introduced in Theorem $\ref{Theorem:Worch}$.
\end{corollary}
\begin{proof} By Proposition $\ref{Proposition:Technical}$ the element $(a-1,b-(r-1)) \in \Delta^{\im}_+(\Gamma_r)$ is an imaginary root and each indecomposable representation with this dimension vector is regular by $(\blacklozenge)$.
Let $\delta \colon 0 \to X_\alpha \to E \to M \to 0$ be a short exact sequence. We assume that $E$ is not indecomposable. Since $M$ and $X_\alpha$ are regular, we know that $E$ is a regular representation as regular representations are closed under extensions \cite[VIII.2.13]{Assem1}. Hence we find $l \geq 2$ and indecomposable regular representations $E_1,\ldots,E_l$ such that $E_1 \oplus \ldots \oplus E_l = E$. Since each $E_i$ is regular, we know that $\dimu E_i = (a_i,b_i) \in \NN^2$ for all $i \in \{1,\ldots,l\}$. According to Theorem $\ref{Theorem:Worch}$, we have $E \not\in \EKP(r)$ and Proposition $\ref{Proposition:Filtration}$ implies that there is exactly one $j \in \{1,\ldots,n\}$ such that $q_{\Gamma_r}(a_j,b_j) + b_j - a_j < 1$. Since $l \geq 2$ we find $i \in \{1,\ldots,n\}$ such that $q_{\Gamma_r}(a_i,b_i) + b_i - a_i \geq 1$ and Proposition $\ref{Proposition:Wiedemann}$ implies that $E_i \in \EKP(r)$. Hence $\Hom_{\Gamma_r}(X_\alpha,E_i) = 0$ and Lemma $\ref{Lemma:NonSplit}$ implies that $\delta$ splits, a contradiction.
\end{proof}

\begin{proposition}\label{Proposition:NotMaximalRank}
Let $(a,b) \in \Delta^{\im}_+(\Gamma_r)$ be an imaginary root such that $a \leq b = \lfloor a L_r \rfloor$ and $q_{\Gamma_r}(a,b) + b -a \leq 0$, then $(a,b)$ does not have the equal kernels property.
\end{proposition}
\begin{proof}
We prove the statement by induction on $a \geq 1$, the case $a = 1$ being trivial, since then $b = r - 1$ and for each $\alpha \in k^r \setminus \{0\}$ we have $\dimu X_\alpha = (a,b)$ with $X_\alpha \not\in \EKP(r)$ by $\ref{Theorem:Worch}$. Now assume that $1 < a$, then Proposition $\ref{Proposition:Technical}$ implies that $(x,y) := \Phi_r(a-1,b-(r-1))$ is an imaginary root such that $q_{\Gamma_r}(x,y) + y - x \leq 0$. We claim that $(x,y) \not\in \underline{\EKP}(r)$. We consider different cases.\\
We assume first that $x \leq y$ and $y < \lfloor x L_r \rfloor$ holds, then Theorem $\ref{Theorem:NotMaximalRank1}$ yields an indecomposable representation not in $\EKP(r)$ with dimension vector $(x,y)$. \\
Now we assume that $x \leq y = \lfloor x L_r \rfloor$.  We have $ry - x = \Phi^{-1}_r(x,y)_1 = a-1$ and therefore $x \leq x(r-1) \leq ry-x = a-1$. By induction we find an indecomposable representation not in $\EKP(r)$ with dimension vector $(x,y)$.\\
If $x > y$, then every indecomposable representation with dimension vector $(x,y)$ is not in $\EKP(r)$.\\ It follows that $(x,y) \not \in \underline{\EKP}(r)$. 
Let $Y$ be an indecomposable regular representation with dimension vector $(x,y)$ such that $Y \notin \EKP(r)$. In view of Theorem $\ref{Theorem:Worch}$ we find $\alpha \in k^r \setminus \{0\}$ such that $0 \neq \Hom_{\Gamma_r}(X_\alpha,Y)$. Now the Auslander-Reiten formula yields $0 \neq \dim_k \Hom_{\Gamma_r}(X_\alpha,Y) = \dim_k \Ext^1_{\Gamma_r}(\tau^{-1}_{\Gamma_r} Y,X_\alpha)$. Hence we find a non-split exact sequence
\[ 0 \to X_\alpha \to E \to \tau^{-1}_{\Gamma_r} Y \to 0.\]
Since $\dimu \tau^{-1}_{\Gamma_r} Y = (a-1,b-(r-1))$, we conclude with Corollary $\ref{Corollary:ExistenceNonSplit}$ that $E$ is indecomposable with dimension vector $(a,b)$. By construction $0 \neq \Hom_{\Gamma_r}(X_\alpha,E)$ and therefore $E \not\in \EKP(r)$ by Theorem $\ref{Theorem:Worch}$.
\end{proof}

\begin{Theorem}\label{Theorem:B}
Let $r \geq 1$ and $\delta \in \Delta_+(\Gamma_r)$ be a positive root. The following statements are equivalent:
\begin{enumerate}
\item [$(i)$] $\delta \in \underline{\EKP}(r) \cup \underline{\EIP}(r)$.
\item [$(ii)$] $q_{\Gamma_r}(\delta) + |\delta_1 - \delta_2| \geq 1$.
\end{enumerate}
\end{Theorem}
\begin{proof}
For $r \in \{1,2\}$ the statement follows from Corollary $\ref{Corollary:Wiedemann}$. Now assume that $r \geq 3$, then Proposition $\ref{Proposition:NotMaximalRank}$ and Theorem $\ref{Theorem:NotMaximalRank1}$ imply that $(i) \implies (ii)$ holds and $(ii) \implies (i)$ follows from Theorem $\ref{Proposition:Wiedemann}$.
\end{proof}

\section{Applications}

\subsection{Elementary abelian groups}

\noindent Let $\Char(k) = p \geq 2$ and $E_r$ be a $p$-elementary abelian group of rank $r \in \NN$. Given $M \in \modd kE_r$ we define $m_i := \dim_k \Rad^i_{kE_r}(M)$ for all $i \in \{0,1,2\}$. As a consequence of Theorem $\ref{Theorem:B}$ we get:

\begin{corollary}
Let $M \in \modd kE_r$ and assume that $M/\Rad^2_{kE_r}(M)$ is indecomposable such that
\[q_{\Gamma_r}(m_0-m_1,m_1-m_2) + m_0 - 2m_1 + m_2 \geq 1.\] 
Then $M$ has the equal images property.
\end{corollary}
\begin{proof}
In view of \cite[5.6.4]{Be1} it sufficies to show that $M/\Rad^2_{kE_r}(M)$ has the equal images property. Since $p \geq 2$, the indecomposable module $M/\Rad^2_{kE_r}(M)$ of Loewy length $\leq 2$ corresponds to an indecomposable representation $N \in \rep(\Gamma_r)$ with dimension vector $(m_0-m_1,m_1-m_2)$. By assumption and Theorem $\ref{Theorem:B}$ we know that $N \in \EIP(r)$. Hence $M/\Rad^2_{kE_r}(M)$ has the equal images property by \cite[2.3]{Wor1}. 
\end{proof}

\subsection{Regular components for wild Kronecker quivers}\label{Section:Elementary}
\noindent We assume throughout that $r \geq 3$, i.e. $\rep(\Gamma_r)$ is a wild category. 
 
\begin{Definition}
A non-zero regular representation $E \in \rep(\Gamma_r)$ is called \textsf{elementary}, if there is no short exact sequence $0 \to A \to E \to B \to 0$ with $A$ and $B$ regular non-zero. We denote the set of all elementary representations by $\El(r) \subseteq \rep(\Gamma_r)$ and set $\cX := \El(r) \cap \EKP(r)$. We let $\cE(\cX)$ be the class of all regular representations that have an $\cX$-filtration.
\end{Definition}

\noindent Note that each elementary representation is, by definition, indecomposable  and each regular representation $M$ has a $\El(r)$-filtration. Usually such a filtration is not uniquely determined and not much is known about the elementary representations that appear as filtrations factors for $M$. If $\dimu M = (a,b)$ satisfies $1 \leq a \leq b = \lfloor a L_r \rfloor$, we get the following restrictions: 

\begin{corollary}\label{Corollary:FiltrationElementary}
Let $M$ be an indecomposable regular representation with $\dimu M = (a,b)$ such that $a \leq b = \lfloor a L_r \rfloor$. Let $0 = M_0 \subset M_{1} \subset \cdots  \subset M_{n-1} \subset M_n = M$ be a filtration such that each $M_i/M_{i-1}$ is an elementary representation.  The following statements hold.
\begin{enumerate}
\item[$(1)$] There is at most one $i \in \{1,\ldots,n\}$ such that $E := M_i/M_{i-1}$ does not have the equal kernels property. In this case there exists $\alpha \in k^r \setminus \{0\}$ such that  $E \cong X_\alpha$.
\item[$(2)$] For all $l \in \NN$ and each filtration $0 = N_1 \subset N_2 \subset \cdots N_{m-1} \subset N_m = \tau^{-l}_{\Gamma_r} M$ with elementary filtration factors we have $N_i/N_{i-1} \in \cX$ for all $i \in \{1,\ldots,m\}$.
\item[$(3)$] For all $l \in \NN$ we have $\tau^{-l}_{\Gamma_r} M \in \cE(\cX)$.

\end{enumerate} 
\end{corollary}
\begin{proof}
Assume that there is some $i \in \{1,\ldots,n\}$ such that $E := M_i/M_{i-1}$ is not in $\EKP(r)$. Then $\ref{Proposition:Filtration}$ implies that $i$ is unique with this property. Let $\dimu E = (u,v)$, then $u \leq v = \lfloor u L_r \rfloor$ by $\ref{Lemma:Average}$. We conclude with \cite[2.5]{Bi3} that $1 \leq u \leq v \leq r - 1$. Recall from $(\star)$ that $r > L_r > r - 1$. Hence $1 = u$ and $v = r - 1$. Now Theorem $\ref{Theorem:Worch}$ implies that $E \cong X_\alpha$ for some $\alpha \in k^r -0$. This proves $(1)$. For $(2)$ just note that every $\El(r)$-filtration of $\tau^{-l}_{\Gamma_r} M$ induces an $\El(r)$-filtration of $M$ by applying $\tau^l_{\Gamma_r}$ to the filtration. Since $\EKP(r)$ is closed under $\tau^{-1}_{\Gamma_r}$ $($see \cite[2.7]{Wor1}$)$ and $\dimu \tau^{-1}_{\Gamma_r} X_\alpha \in \underline{\EKP}(r)$ the claim follows. Now $(3)$ is just a special case of $(2)$.
\end{proof}

\begin{examples}
In the following we consider $r= 3$.
\begin{enumerate}
\item[(i)] Consider an indecomposable representation $N$ with  dimension vector $(3,6)$  and the equal kernels property $($see \cite[4.1]{Bi3}$)$. Let $E$ be an elementary representation with dimension vector $(a,b)$ and $a \leq 3, b \leq 6$. Westwick's Theorem and \cite[Theorem 1]{Ri10} yield $(a,b) \in \{(2,4),(2,5)\}$. Hence $N$ does not have an $\cX$-filtration, although $N \in \EKP(3)$.
\item[(ii)] Consider the indecomposable representation $M \in \rep(C_3)$:

\[ \xymatrix{
k& &k & &k & &k \ar[ld]^{[1 0]}  & \\
k& k \ar[lu] \ar[l] \ar[ur] & & k \ar[lu] \ar[ru] \ar^{[1 1]}[rr]  & & k^2 & k. \ar[l]^{[0 1]}& 
}
\]

\noindent There is a short exact sequence $0 \to X \to M \to Y \to 0$ in $\rep(C_3)$ such that $X, Y$ are indecomposable with $\dimu \pi_{\lambda}(X) = (2,5)$ and $\dimu \pi_\lambda(Y) = (2,1)$. Since every indecomposable representation with dimension vector $(2,5)$ or $(2,1)$ is elementary this yields a filtration of $\pi_\lambda(M)$ by elementary $\Gamma_3$-representations. We apply $\tau^{-1}_{\Gamma_3}$ and get a filtration of $\tau^{-1}_{\Gamma_3} \pi_\lambda(M)$ by elementary representations with dimension vectors $(13,34)$ and $(1,2)$. Note that $q_{\Gamma_3}(14,36) + |14-36| = -20+22 = 2 \geq 1$. We conclude with Corollary  $\ref{Corollary:DimensionVectorIsMaximal}$ that $(a,b) := \dimu \tau^{-1}_{\Gamma_3} \pi_\lambda(M) = (14,36)$ satisfies $a \leq b = \lfloor a L_3 \rfloor$. Hence not every filtration of $\pi_\lambda(M)$ is an $\cX$-filtration.
\item[(iii)] Let $M$ be an indecomposable representation with dimension vector $(18,47)$. Since $(a,b) := \dimu \tau_{\Gamma_3} M = (3,7)$ and $7 = \lfloor 3 L_3 \rfloor$, every filtration of $M$ with elementary filtration factors is an $\cX$-filtration by $\ref{Corollary:FiltrationElementary}(2)$.

\end{enumerate} 
\end{examples}

\begin{Lemma}\label{Lemma:Orbit} Let $(a,b) \in \Delta^{\im}_+(\Gamma_r)$. The following statements hold.
\begin{enumerate}
\item If $a \geq b$, then $\Phi_r(a,b)_1 - \Phi_r(a,b)_2 \geq a - b + (r^2-2r)a > a - b \geq 0$.
\item If $a \leq b$, then $\Phi_r^{-1}(a,b)_2 - \Phi_r^{-1}(a,b)_1 \geq b -a + (r^2-2r) b > b - a \geq 0$.
\end{enumerate}
\end{Lemma}
\begin{proof}
\begin{enumerate}
\item We have $\Phi_r(a,b)_1 - \Phi_r(a,b)_2 + b - a= a(r^2-r-2) - b(r-2) \geq a(r^2-2r)$. Since $r \geq 3$ and $a \geq 1$, we get $a(r^2-2r) > 0$.
\item The follows by duality.
\end{enumerate}
\end{proof}

\noindent Recall from \cite[XIII.1.1]{Assem3} that for each $\delta \in \Delta^{\im}_+(\Gamma_r)$ with Coxeter orbit $\cO := \langle \Phi_r \rangle \delta$, the map 
\[ \varphi_\delta \colon \ZZ \to \cO, l \mapsto \Phi^{l}_r \delta \]
is a bijection. Thanks to Theorem $\ref{Theorem:B}$ we have:

\begin{corollary}\label{Corollary:Coxeterorbits}
Let $\cO$ be the Coxeter orbit of an imaginary root. There exist uniquely determined elements $\delta_{\cO} \in \cO$ and $m_{\cO} \in \NN_0$ such that 
\begin{enumerate}
\item[$(i)$] $\underline{\EIP}(r) \cap \cO = \{ \Phi^l_r \delta_{\cO} \mid l \in \NN\}$ and
\item [$(ii)$] $\underline{\EKP}(r) \cap \cO = \{ \Phi^{-l}_r \delta_{\cO} \mid l \geq m_\cO\}$.
\end{enumerate} 
\end{corollary}

\begin{proof}
By Lemma $\ref{Lemma:Orbit}$ we find $(a,b) \in \cO$ such that $a \geq b$, $d_l := \Phi^{l}_r (a,b)_1 - \Phi^{l}_r (a,b)_2 \geq 0$ and $b_{l} := \Phi^{-l}_r (a,b)_2 - \Phi^{-l}_r (a,b)_1 \geq  0$ for all $l \in \NN$. Lemma $\ref{Lemma:Orbit}$ also implies that $(d_n)_{n \in \NN_0}$ and $(b_n)_{n \in \NN}$ are strictly increasing sequences. Since $q_{\Gamma_r}(-)$ is constant on $\cO$ we find $n \in \NN_0$ and $l \in \NN$ minimal such that $q_{\Gamma_r}(a,b)+ d_n \geq 1$ and $q_{\Gamma_r}(a,b)+ b_l \geq 1$. We set $\delta_{\cO} := \Phi^{n-1}_r (a,b)$ and $m_{\cO} := l$.
\end{proof}

\begin{Remark}
The existence of $m_\cO$ can also be proved by means of Theorem $\ref{Theorem:Worch}$ and \cite[4.6]{Ker3}. Our proof, however, provides an algorithm to compute $m_{\cO}$. 
\noindent As an application we get computable bounds for the invariants $\rk(\cC)$, $\cW(\cC)$, introduced in \cite{Ker1} and \cite{Wor1}, attached to a regular component $\cC$ of the Auslander-Reiten quiver of $\Gamma_r$. It follows from \cite[1.7]{Ker1} and \cite[3.4]{Wor1} that $-1 \leq -\rk(\cC) \leq \cW(\cC)$.
\end{Remark}

\begin{corollary}
Let $r \geq 3$ and $\cC$ be a regular component of the Auslander-Reiten quiver of $\cC$. Let $M \in \cC$ be an indecomposable representation of quasi-length $\ql(M) \in \NN$. Let $\cO_M$ be the Coxeter orbit of $\dimu M$, then $\cW(\cC) \leq m_{\cO_M} - \ql(M) + 1$.
\end{corollary}
\begin{proof}
For each $l \in \NN$ we denote by $d(l)$ the distance of the cones $\EIP(r) \cap \cC$ and $\EKP(r) \cap \cC$ $($see \cite[3.1]{Wor1}$)$ at the level of quasi-length $l$. We have $d(l) = \cW(\cC) + (l-1)$ for all $l \in \NN$ by definition of $\cW(\cC)$. By definition of $m_{\cO_M}$ in Corollary $\ref{Corollary:Coxeterorbits}$, we have $d(\ql(M)) \leq m_{\cO_M}$. We conclude  
$\cW(\cC) = d(\ql(M)) - \ql(M) + 1\leq m_{\cO_M} - \ql(M) + 1$.
\end{proof}

\begin{example}
We consider $r = 3$ and the Coxeter orbit $\cO$  of $(30,31)$. We have $q_{\Gamma_3}(30,31) = - 929$ and the orbit looks as follows:
\footnotesize
\[
\cdots \stackrel{\Phi_3}{\leftarrow} (6846,2615)  \stackrel{\Phi_3}{\leftarrow}  (999,382) \stackrel{\Phi_3}{\leftarrow} (147,59) \stackrel{\Phi_3}{\leftarrow}  (30,31) \stackrel{\Phi_3}{\leftarrow} (63,158) \stackrel{\Phi_3}{\leftarrow} (411,1075) \stackrel{\Phi_3}{\leftarrow} (2814,7367)   \stackrel{\Phi_3}{\leftarrow} \cdots
.\]
\normalsize
We conclude that $\delta_\cO = (999,382)$ and $m_{\cO} = 5$. By \cite[3.4]{BoChen1}, every indecomposable representation $M$ with dimension vector $(30,31)$ is quasi-simple. Let $\cC$ be the regular component containing $M$, then $\cW(\cC) \leq m_\cO - \ql(M) + 1  = 5$.
\end{example}

\noindent In the following we consider for $r \geq 3$ the sequence $(A_i(r))_{i \in \NN}$ given by $A_1(r) := 1$, $A_2(r): = r$ and $A_{i+2}(r) := r A_{i+1}(r) - A_i(r)$ for all $i \in \NN$. 

\begin{corollary}
Let $(a,b) \in \Delta^{\im}_+(\Gamma_r)$ be an imaginary root such that $a \leq b$. There exists $t \geq r$ such that for all $s \geq t$ the following statements hold:
\begin{enumerate}
\item[$(1)$] We have $\Phi^{-1}_{s}(a,b) \in \underline{\EKP}(s)$. 
\item[$(2)$] Each indecomposable representation $M \in \rep(\Gamma_s)$ with dimension vector $(a,b)$ is regular and quasi-simple.
\item[$(3)$] For each indecomposable representation $M \in \rep(\Gamma_s)$ with dimension vector $(a,b)$ the almost split sequence \[ 0 \to M \to E \to \tau^{-1}_{\Gamma_s} M \to 0\] 
has indecomposable middle term and $\tau^{-1}_{\Gamma_s} M \in \EKP(s)$. 
\end{enumerate}

\end{corollary}
\begin{proof}
We set $t := \max\{b+2,r+1\}$ and fix $s \geq t$. \\
$(1)$ Cleary $(a,b)$ is an imaginary root for $q_{\Gamma_s}$ since $q_{\Gamma_s}(a,b) \leq q_{\Gamma_r}(a,b) \leq 0$. Let $\delta := \Phi^{-1}_{\Gamma_s}(a,b)$. Then $\delta_1 \leq \delta_2$ and
\begin{align*}
q_{\Gamma_s}(\delta) + \delta_2 - \delta_1 &=  a^2 + b^2 - sab + (s^2-1)b-sa - (bs-a) \\
&= a^2+b^2 +(s^2-s-1)b - (sb+s-1)a \geq a^2 + b^2 \geq 1,
\end{align*}
since $s^2-s -1 \geq sb+s-1$ by the choice of $s \geq t \geq b+2$. We conclude with Theorem $\ref{Proposition:Wiedemann}$ that $\Phi^{-1}_s(a,b) = \delta \in \underline{\EKP}(s)$.\\
$(2)$ Let $M \in \rep(\Gamma_s)$ be indecomposable with $\dimu M = (a,b) \in \Delta^{\im}_+(\Gamma_s)$, then $M$ is regular by $(\blacklozenge)$. We consider the ascending sequence $(A_i(s))_{i \in \NN}$. In view of \cite[3.4]{BoChen1} it suffices to show that $A_i(s)$ is not a common divisor of $a$ and $b$ for all $i \in \NN_{\geq 2}$. But this is trivial since $A_2(s) = s \geq t > b$.\\
$(3)$ This follows immediatly from $(1)$ and $(2)$.
\end{proof}

\noindent The proof of the following result may be found in the appendix.

\begin{Lemma}\label{Lemma:DimensionShiftSum}
Let $(a,b) \in \Delta_+^{\im}(\Gamma_r)$ and $l \in \NN$. The following statements hold:
\begin{enumerate}
\item[$(i)$] The element $(x,y) := \sum^{l}_{i=0} \Phi^{-i}_{r} (a,b)$ satisfies $x L_r - y = \frac{A_{l+1}(r)}{L^l_r}(aL_r - b)$.
\item[$(ii)$] We have $\frac{A_{l}(r)}{L^{l}_r} < 1$.
\end{enumerate}
\end{Lemma}

\begin{proposition}
Let $\cC$ be a regular component. There exists a uniquely determined quasi-simple representation $X_\cC$ in $\cC$ such that for each $N \in \cC$ the following statements are equivalent:
\begin{enumerate}
\item[$(1)$] Each $\El(r)$-filtration of $N$ is an $\cX$-filtration.
\item[$(2)$] The representation $N$ is a successor of $X_\cC$.
\end{enumerate}
\end{proposition}
\begin{proof} 
We first show that there exists a quasi-simple representation $Y$ in $\cC$ such that for every successor $N$ of $Y$, every $\El(r)$-filtration of $N$ is an $\cX$-filtration. 
By Corollary $\ref{Corollary:DimensionVectorIsMaximal}$ and Corollary $\ref{Corollary:Coxeterorbits}$ we find a quasi-simple representation $X \in \cC$ such that $\dimu X = (u,v)$ satisfies $u \leq v = \lfloor u L_r \rfloor$. We set $Y := \tau^{-2}_{\Gamma_r} X$ and $(a,b) := \dimu \tau^{-1}_{\Gamma_r} X$. Since $u L_r - v < 1$, we get
\begin{align*}
a L_r - b &= v(r L_r - r^2 + 1) + u(r - L_r) \stackrel{(\star)}{=} v(r L_r - r^2 + 1) + \frac{u}{L_r} \\
&=\frac{1}{L^2_r}(uL_r - v) + v(\frac{1}{L^2_r}+ r L_r - r^2 +1) \stackrel{(\star)}{=}  \frac{1}{L^2_r}(uL_r - v) < \frac{1}{L^2_r}.
\end{align*} 
Let $l \in \NN$ and note that
\[(x,y) := \sum^{l-1}_{i=0} \Phi^{-i}_r (a,b) = \dimu \tau^{-1}_{\Gamma_r} X[l].\] 
Clearly $x \leq y$ and $(\blacktriangledown)$ implies $0 \leq xL_r -y$. Now we conclude with Lemma $\ref{Lemma:DimensionShiftSum}(i),(ii)$:
\[ 0 \leq xL_r -y = \frac{A_{l}(r)}{L^{l-1}_r}(aL_r - b) < \frac{A_{l}(r)}{L^{l+1}_r} < 1.\]
Hence $x \leq y = \lfloor x L_r \rfloor$. Therefore Corollary $\ref{Corollary:FiltrationElementary}$ implies that for $i \in \NN$ every $\El(r)$-filtration of $\tau^{-i} (\tau^{-1}_{\Gamma_r} X)[l] = \tau^{-{(i-1)}}_{\Gamma_r}(Y[l])$ is an $\cX$-filtration. Hence $Y$ has the desired property.\\
Clearly $Y \in \EKP(r)$ and we find $m \in \NN$ such that $\tau^{m}_{\Gamma_r} Y \not\in \EKP(r)$. Now let $l \in \{0,\ldots,m-1\}$ maximal such that for every successor $N$ of $\tau^l_{\Gamma_r} Y$, every $\El(r)$-filtration of $N$ is an $\cX$-filtration.  Set $X_\cC := \tau^l_{\Gamma_r} Y$ and note that $X_\cC$ has the desired properties.
\end{proof}

The duality $D_{\Gamma_r} \colon \rep(\Gamma_r) \to \rep(\Gamma_r)$ introduced in \cite[2.2]{Wor1} satisfies $D_{\Gamma_r}(\EKP(r)) = \EIP(r)$ and commutes with the Auslander-Reiten translation. Therefore we conclude:

\begin{proposition}
Let $\cC$ be a regular component. There exists a uniquely determined quasi-simple representation $Y_\cC$ in $\cC$ such that for each $N \in \cC$ the following statements are equivalent:
\begin{enumerate}
\item[$(1)$] Each $\El(r)$-filtration of $N$ is a $\cY$-filtration, where $\cY := \El(r) \cap \EIP(r)$.
\item[$(2)$] The representation $N$ is a predecessor of $Y_\cC$.
\end{enumerate}
\end{proposition}

\subsection{Non-split exact sequences in the universal covering}

Given an indecomposable representation $M \in \rep(C_r)$, the underlying graph of $\supp(M)$ forms a finite tree.\\
The representation $M$ is called \textsf{sink representation} $($\textsf{source representation}$)$, provided each path in $\supp(M)$ of maximal length connects two leaves of $M$ that are sinks $($sources$)$. If each path in $\supp(M)$ of maximal length connects a source leaf and a sink leaf, $M$ is called a \textsf{flow representation}.
It follows from \cite[2.4]{Ri11} that every regular indecomposable representation is a sink, source or a flow representation and \cite[Theorem 4]{Ri11} gives a description of the distribution of sink, source and flow representations in a given regular component $\cC$. These considerations lead to new invariants for gradable indecomposable representations in $\rep(\Gamma_r)$, i.e. indecomposable representations in the essential image of $\pi_{\lambda} \colon \rep(C_r) \to \rep(\Gamma_r)$.\\
By Theorem $\ref{Theorem:INJEKP}$ we know that every indecomposable representation $N \in \Inj(r)$ is a sink representation. Moreover, we have $\dim_k N_{s(\gamma)} \leq \dim_k N_{t(\gamma)}$ for each arrow $\gamma \in (C_r)_1$. In contrast to the whole class of sink representations, there is a natural analogue of $\Inj(r)$ for ungraded Kronecker representations given by the equal kernels property.\\
Therefore it makes sense to take a closer look at the structure of $\supp(M)$ and $\dimu M$ for $M \in \Inj(r)$ indecomposable in hope to get a better understanding of sink representations. In view of Theorem $\ref{Theorem:INJEKP}$ we study the following question:

\vspace*{6pt}

\noindent \textbf{Question} Let $r \geq 1$ and $M \in \rep(C_r)$ be an indecomposable representation such that $\dim_k M_{s(\gamma)} \leq \dim_k M_{t(\gamma)}$ for each arrow $\gamma \in (C_r)_1$. Does this already imply that $M \in \Inj(r)$?

\vspace*{6pt}
Let us consider the first non-trivial case, i.e. $r = 2$. Then the underlying graph of $C_2$ is of type $\mathbb{A}^\infty_\infty$, i.e. each indecomposable representation $N$ in $\rep(C_2)$ can be considered as a representation for a quiver $Q$ with underlying graph $\mathbb{A}_n$ for some $n \in \NN$. Therefore  $N$ is a thin representation. Hence every indecomposable representation $M \in \rep(C_2)$ with $\dim_k M_{s(\gamma)} \leq \dim_k M_{t(\gamma)}$ for all $\gamma \in (C_2)_1$ is already in $\Inj(2)$ and $\pi_{\lambda}(M) \in \EKP(2)$.

\vspace*{6pt}

In the following we give a negative answer to the question for $r \geq 3$. The main ingredient in our construction of counterexamples are indecomposable representations $X_i \in \rep(C_r)$, $\{1,\ldots,r\}$ with corresponding push-down $\pi_{\lambda}(X_i) = X_{e_i}$, where $e_i$ denotes the $i$-th canonical basis vector of $k^r$ $($see Theorem $\ref{Theorem:Worch})$.\\ 
We fix a source $z \in (C_r)_0$ and denote by $y_i \in n_{C_r}(z)$ the unique element with $\pi(z \to y_i) = \gamma_i$. Then $X_{i,z} = X_i$ is by definition the uniquely determined indecomposable and thin representation with $\supp(X_i) = \{z\} \cup \bigcup_{j \neq i} \{y_j\}$.\\
Let $M \in \rep(C_r)$ be an indecomposable representation, $x \to y \in (C_r)_1$, and $i \in \{1,\ldots,r\}$ with $\pi(x \to y) = \gamma_i$. Let $g \in G(r)$ be the unique element with $x \in \supp(X^g_i)$. In view of $\ref{Theorem:Worch}$, $\ref{TheoremRingelGabriel}$ and $\ref{Theorem:INJEKP}$ we have $\Hom_{C_r}(X^g_i,M) = 0$ if and only if $M(x \to y)$ is injective and $\Ext^1_{C_r}(X^g_i,M) = 0$ if and only if $M(x \to y)$ is surjective. Now we can prove the following:

\begin{Lemma}\label{Lemma:ExtensionUniversal}
Let $M \in \rep(C_r)$ be an indecomposable representation in $\Inj(r)$ and $\gamma \in (C_r)_1$ such that $\tau_{C_r} M(\gamma)$ is surjective but not injective. Let $i \in \{1,\ldots,r\}$ be the unique element with $\pi(\gamma) = \gamma_i$. There exists $g \in G(r)$ and a short exact sequence $0 \to X^g_i \to E \to M \to 0$ with indecomposable middle term $E$. In other words, there exists an indecomposable representation $\pi_{\lambda}(E) \not\in \EKP(r)$ with dimension vector $\dimu \pi_\lambda(E)+(1,r-1)$. If in addition $s(\gamma) \in \supp(M)$, then $\supp(M) = \supp(E)$.
\end{Lemma}
\begin{proof} 
Since $M \in \Inj(r)$, we have $\Hom_{C_r}(X^h_i,M) = 0$ for all $h \in G(r)$. Let $g \in G(r)$ be such that $s(\gamma) \in \supp(X^g_i)$. Since $\tau_{C_r} M(\gamma)$ is surjective, we conclude with the Auslander-Reiten formula
\[ \Hom_{C_r}(M,X_i^g) \cong \Ext^1_{C_r}(X_i^g,\tau_{C_r} M) = 0.\]
Hence $M$ and $X_i^g$ are $\Hom$-orthogonal.
Since $\tau_{C_r} M(\gamma)$ is not injective, we have 
\[0 \neq \Hom_{C_r}(X_i^g,\tau_{C_r} M) \cong \Ext^1_{C_r}(M,X_i^g).\]
Hence \cite[3.7]{Weist} implies that we can find a short exact sequence 
\[ 0 \to X_i^g \to E \to M \to 0 \]
with indecomposable middle term.\\
Now we assume that $s(\gamma) \in \supp(M)$. Since $M \in \Inj(r)$, we get $\supp(X^g_i) \subseteq \{s(\gamma)\} \cup n_{C_r}(s(\gamma)) \subseteq \supp(M)$. 
Hence $\supp(E) = \supp(X^g_i) \cup \supp(M) = \supp(M)$.
\end{proof}

\begin{example}
Let $r = 3$, fix a source $x \in (C_3)_0$ and consider an indecomposable regular representation $Y \in \rep(C_r)$ with support $\supp(Y) = \{x\} \cup n_{C_3}(x)$ such that $\dim_k Y_x = 2$ and $\dim_k Y_z = 1$ for all $z \in n_{C_3}(x)$. We illustrate the support of $Y$ in the following diagram $(x$ and a fixed neighbour $y$ of $x$ are indicated by small indices$)$: 
\tikzstyle{level 1}=[sibling angle=120,level distance = 30, ->]
\tikzstyle{level 2}=[sibling angle=90,level distance = 30, <-]
\tikzstyle{level 3}=[sibling angle=60,level distance = 30, ->]
\tikzstyle{level 3}=[sibling angle=45,level distance = 30, ->]
\tikzstyle{every node}=[]
\tikzstyle{edge from parent}=[segment angle=10,draw]

\begin{center}
\begin{tikzpicture}[grow cyclic, shape=circle,cap=round]
\node {$2_x$} 
    child { node {1}}
    child { node {$1_y$}}
    child { node {1}};
\end{tikzpicture}
\end{center}
The support of $M := \tau^{-1}_{C_3} Y$ looks as follows:
\begin{center}
\begin{tikzpicture}[grow cyclic, shape=circle,cap=round]
\node {$1_x$} 
    child { node {2} child {node {1} child {node {1}} child {node {1}}} child {node {1} child {node {1}} child {node {1}}}}
    child { node {$2_y$} child {node {1} child {node {1}} child {node {1}}} child {node {1} child {node {1}} child {node {1}}}}
    child { node {2} child {node {1} child {node {1}} child {node {1}}} child {node {1} child {node {1}} child {node {1}}}};
\end{tikzpicture}
\end{center}
Note that $M \in \Inj(3)$, since every source in $\supp(M)$ is thin and $M$ is indecomposable. Now consider $i \in \{1,2,3\}$ such that $\pi(x \to y) = \gamma_i$. Without loss we can assume that $x \in \supp(X_i)$. Then the support of $X_i$ looks as follows:
\begin{center}
\begin{tikzpicture}[grow cyclic, shape=circle,cap=round]
\node {$1_x$} 
    child { node {1}}
    child { node {$0_y$}}
    child { node {1}};
\end{tikzpicture}
\end{center}
Since $Y(x \to y)$ is surjective but not injective, Lemma $\ref{Lemma:ExtensionUniversal}$ implies that we find a short exact sequence $0 \to X_i \to E \to M \to 0$ with indecomposable middle term. Hence the support of $E$ looks as follows: 
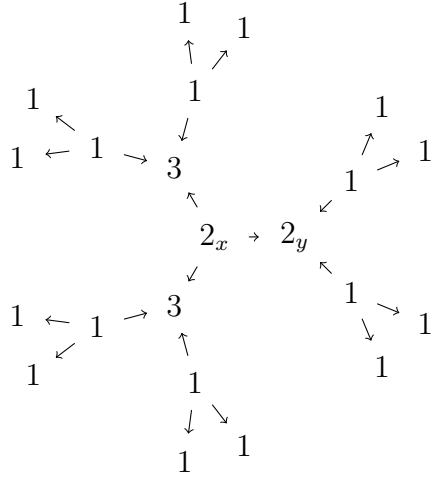
\begin{center}
\begin{tikzpicture}[grow cyclic, shape=circle,cap=round]
\node {$2_x$} 
    child { node {3} child {node {1} child {node {1}} child {node {1}}} child {node {1} child {node {1}} child {node {1}}}}
    child { node {$2_y$} child {node {1} child {node {1}} child {node {1}}} child {node {1} child {node {1}} child {node {1}}}}
    child { node {3} child {node {1} child {node {1}} child {node {1}}} child {node {1} child {node {1}} child {node {1}}}};
\end{tikzpicture}
\end{center}
Observe that $E_a \leq E_b$ for all arrows $a \to b$ but $E \notin \Inj(3)$ since $\ker E(x \to y) \neq 0$.
Clearly we can extend this example to all $r \geq 3$: We fix a source $x \in (C_r)_0$ and consider an indecomposable regular representation $Y \in \rep(C_r)$ with support $\supp(Y) = \{x\} \cup n_{C_r}(x)$ such that $\dim_k Y_x = r-1$ and $\dim_k Y_z = 1$ for all $z \in n_{C_r}(x)$. As before we fix a neighbour $y$ of $x$. Then $M:=\tau^{-1}_{C_r} Y \in \Inj(r)$, $\dim_k M_x = 1$ and $\dim_k M_z = r-1$ for all $z \in n_{C_r}(x)$. Let $i \in \{1,\ldots,r\}$ such that $\pi(x \to y) = \gamma_i$ and apply Lemma $\ref{Lemma:ExtensionUniversal}$ to find a short exact sequence $0 \to X_i \to E \to M \to 0$ with indecomposable middle term. Since $0 \neq \ker E(x \to y)$, we conclude as before that $E \not\in \Inj(r)$ and $\pi_{\lambda}(E) \not\in \EKP(r)$. By construction we have $\dim_k E_a \leq \dim_k E_b$ for all arrows $a \to b \in (C_r)_1$.

\end{example}

\begin{Remark}
Note that these counterexamples are minimal in the sense that there is exactly one arrow $\gamma \in (C_r)_0$ such that $E(\gamma)$ is not injective.
\end{Remark}
 \section{Appendix}

\footnotesize
\noindent In the following we provide the proofs of Lemma $\ref{Lemma:DistanceToCa}$, Lemma $\ref{Lemma:Average}$, Proposition $\ref{Proposition:Technical}$ and Lemma $\ref{Lemma:DimensionShiftSum}$.

\bigskip

\noindent \textbf{Proof of Lemma $\ref{Lemma:DistanceToCa}$}
\begin{proof}
By assumption we have $a^2 + b^2 - rab + b -a \leq 0$ and conclude
$b^2+b(1-ra) \leq \frac{4a-4a^2}{4}$. This is equivalent to
\[(b+\frac{1-ra}{2})^2 \leq \frac{4a-4a^2+1-2ra+r^2a^2}{4} = \frac{(r^2-4)a^2+(4-2r)a+1}{4},\]
in particular $ \frac{(r^2-4)a^2+(4-2r)a+1}{4} \geq 0$ and therefore
\begin{align*}
& \quad \ \ \frac{1-ra}{2} - \frac{\sqrt{(r^2-4)a^2+(4-2r)a+1}}{2} \leq -b\\
&\Leftrightarrow \frac{1-ra}{2} - \frac{\sqrt{(r^2-4)a^2+(4-2r)a+1}}{2} + \frac{r+\sqrt{r^2-4}}{2}a \leq -b +\frac{r+\sqrt{r^2-4}}{2}a \\
&\Leftrightarrow \frac{1}{2} + \frac{\sqrt{r^2-4}a- {\sqrt{(r^2-4)a^2+(4-2r)a+1}}}{2} \leq a L_r - b.
\end{align*}
Since $r \geq 3$ and $a \geq 1$, we have $\sqrt{(r^2-4)a^2+(4-2r)a+1} \leq \sqrt{(r^2-4)a^2+(4-6)a+1} \leq \sqrt{(r^2-4)a^2}$ and conclude
$a L_r -b \geq \frac{1}{2}.$
\end{proof}

\noindent \textbf{Proof of Lemma $\ref{Lemma:Average}$}

\noindent Since $q_{\Gamma_r}(a_i,b_i) \leq 0$, we conclude with $(\blacktriangledown)$ that $a_i \leq b_i \leq a_i L_r$ or $b_i \leq a_i \leq b_i L_r$ for all $i \in \{1,\ldots,n\}$. Hence $b_i \leq \lfloor a_i L_r \rfloor \leq a_i L_r$ for all $i \in \{1,\ldots,n\}$.
Assume that $b_i < \lfloor a_i L_r \rfloor $ for some $i \in \{1,\ldots,n\}$. We get $\sum^n_{j=1} (b_j+\delta_{ij} -a_j L_r) \leq 0$, hence $\frac{b+1}{a}= \frac{\sum^n_{j=1} b_j + \delta_{ij}}{\sum^n_{j=1} a_j} \leq L_r$ and $a \leq b+1 \leq a L_r$, a contradiction to $(\blacktriangledown)$ since $b = \lfloor a L_r \rfloor$. Hence $b_i = \lfloor a_i L_r \rfloor$ for all $i \in \{1,\ldots,n\}$. In particular, $a_i \leq b_i = \lfloor a_i L_r \rfloor$ for all $i \in \{1,\ldots,n\}$.
\hfill $\square$

\bigskip

\noindent \textbf{Proof of Proposition $\ref{Proposition:Technical}$}

\noindent We divide the proof into several steps. Let $(u,v) \in \NN^2_0$. At first we prove: 
\begin{equation}
(u,v) \in \Delta^{\im}_+(\Gamma_r), u \geq v + r - 1 \implies \Phi^{-1}_r(u,v) + (0,1) \in \Delta^{\im}_+(\Gamma_r).
\end{equation} 
Assume that $(u,v) \in \Delta^{\im}_+(\Gamma_r)$ is an imaginary root such that $u \geq v + r - 1$. Clearly $\Phi_r^{-1}(u,v)$ is an imaginary root. We distinguish two cases.  If $\Phi_r^{-1}(u,v)_1 > \Phi_r^{-1}(u,v)_2$, then the result is trivial. Hence we can assume that $\Phi_r^{-1}(u,v)_1 \leq \Phi_r^{-1}(u,v)_2$. We have $\Phi_r^{-1}(u,v)_1 \leq \Phi_r^{-1}(u,v)_2 + 1$ and $L_r - r < 0$ implies
\begin{align*}
\Phi_r^{-1}(u,v)_2  +  1 - L_r \Phi_r^{-1}(u,v)_1 &=  -ru + (r^2-1)v + 1 - L_r(-u+rv) \\
&= u(-r + L_r) + v(r^2-1-r L_r) + 1 \\
&\leq (v+r-1)(-r+L_r)+v(r^2-1-r L_r) +1 \\
&= v(r^2-r+L_r-r L_r-1) +(-r^2+r-L_r+r L_r+1) = q(v-1),
\end{align*}
where $q:=  r^2-r+L_r-L_r r-1$. By $(\star)$ we have $q = \frac{1}{L_r}(r-1)-1 \leq 0$, and conclude 
\[ \Phi_r^{-1}(u,v)_1 \leq \Phi_r^{-1}(u,v)_2  +  1 \leq L_r \Phi_r^{-1}(u,v)_1.\]
Therefore $(\blacktriangledown)$ implies that $\Phi_r^{-1}(u,v) + (0,1)\in \Delta^{\im}_+(\Gamma_r)$. \hfill $\diamond$\\
Now we show show the following equation: 
\begin{equation}
q_{\Gamma_r}(u-1,v-(r-1)) = q_{\Gamma_r}(u,v) + u(r^2-r-2) + v(-r+2) + 2 - r.
\end{equation}
\noindent We have
\begin{align*}
 q_{\Gamma_r}(u-1,v-(r-1)) &= q_{\Gamma_r}(u,v) + q_{\Gamma_r}(1,r-1) - \langle (u,v),(1,r-1) \rangle_{\Gamma_r} -  \langle(1,r-1),(u,v) \rangle_{\Gamma_r}\\
&= q_{\Gamma_r}(u,v)+2-r- 2(u+v(r-1))+r(u(r-1)+v) \\
&= q_{\Gamma_r}(u,v) + u(r^2-r-2) + v(-r+2) + 2 - r.
\end{align*}
\hfill $\diamond$
\begin{equation}
(x',y') := \Phi_r(u-1,v-(r-1)) \  \text{satisfies} \ q_{\Gamma_r}(x',y') + y' - x' = q_{\Gamma_r}(u,v) + v - u.
\end{equation}
Since $q_{\Gamma_r}$ is invariant under $\Phi_r$, we conclude with $(2)$ that
\begin{align*}
q_{\Gamma_r}(x',y') + y' - x' &= q_{\Gamma_r}(u,v) + u(r^2-r-2) + v(-r+2) + 2 - r \\
&+ r(u-1)-(v-(r-1)) - ((r^2-1)(u-1)-r(v-(r-1))\\
&= q_{\Gamma_r}(u,v) + u(r^2-r-2+r-(r^2-1)) + v(-r+2-1+r)\\
&+ (2-r-r+(r-1)+(r^2-1)-r(r-1))\\
&= q_{\Gamma_r}(u,v) + v - u. 
\end{align*}
\hfill $\diamond$

\noindent Now we consider $(a-1,b-(r-1))$. We have $a - 1 \leq b - (r-1)  \Leftrightarrow r-2\leq b - a$. Since $b = \lfloor a L_r \rfloor$, we conclude with $(\star)$ that $b - a \geq a L_r - 1 - a = (L_r-1)a - 1\geq (r-1-1)a-1 \geq 2(r-2)- 1 = r - 2 + r - 3 \geq r-2$. Hence $a-1 \leq b-(r-1)$, which proves Proposition $\ref{Proposition:Technical}(i)$.  We conclude with $(2)$ and the assumption that
\begin{align*}
q_{\Gamma_r}(a-1,b-(r-1)) &= q_{\Gamma_r}(a,b) + 2 - r + a(r^2-r-2) + b(-r+2)\\
&= q_{\Gamma_r}(a,b) +b - a + 2 - r + a(r^2-r-1) + b(-r+1)\\
&\leq 2 - r + a(r^2-r-1)+ b(-r+1)\\
&\leq 2 - r + \Phi_r(a,b)_1 - \Phi_r(a,b)_2.
\end{align*}
We set $u := \Phi_r(a,b)_1$ and $v := \Phi_r(a,b)_2$. The assumption $u - v > r - 2$ in conjunction with $(1)$ yields that $(a,b+1)$ is an imaginary root. But this is a contradiction since $b = \lfloor a L_r \rfloor$ is maximal. Hence $q_{\Gamma_r}(a-1,b-(r-1)) \leq 0$ and $(a-1,b-(r-1))$ is an imaginary root. Therefore $(x,y) := \Phi_r(a-1,b-(r-1))$ is also an imaginary root. Hence we have established the second statement of Proposition $\ref{Proposition:Technical}$. The third statement follows immediatly from $(3)$.
\hfill $\square$

\bigskip
\setcounter{equation}{0}
\noindent \textbf{Proof of Lemma $\ref{Lemma:DimensionShiftSum}$}

\noindent In the proof we write $A_i$ instead of $A_i(r)$ for all $i \in \NN$. By definition we have $A_1 := 1$, $A_2 := r$ and $A_{l+2} := r A_{l+1} - A_l$ for all $l \in \NN$. We claim that for all $l \in \NN$
\begin{equation}
 A_{l+1} = L^{l}_r - A_{l} L_r + r A_{l}.
\end{equation} The proof is by induction on $l \in \NN$. We have $A_2 = r = L_r - A_1 L_r + r A_1$ and conclude with the inductive hypothesis
\[A_{l+2} = r A_{l+1} - A_l = r A_{l+1} - (\frac{A_{l+1} - L^l_r}{r-L_r}) \stackrel{(\star)}{=} r A_{l+1} - L_r A_{l+1} + L^{l+1}_r.\]
This proves $(1)$. \hfill $\diamond$\\

\noindent Now we prove $(i)$ by induction on $l$. We let $(x_l,y_l) := \sum^{l}_{i=0} \Phi^{-i}_r (a,b)$ and get for $l = 1$ that $(x_l,y_l) = (rb,r^2 b -ra)$. Hence
\[x_1 L_r - y_1 = ra + (r L_r - r^2) b \stackrel{(\star)}{=} ra - r \frac{1}{L_r} b = \frac{r}{L_r}(a L_r -b) = \frac{A_2}{L_r}(aL_r -b).\]
We have $(x_{l+1},y_{l+1}) = \sum^{l+1}_{i=0} \Phi^{-i}_r (a,b) = (a,b) + \sum^{l}_{i=0} \Phi^{-i}_r \Phi^{-1}_r(a,b)$. We apply the inductive hypothesis to $\Phi^{-1}_r(a,b) = (rb-a,(r^2-1)b-ra)$ and get
\begin{align*}
x_{l+1} L_r - y_{l+1}  &= a L_r  - b + \frac{A_{l+1}}{L^l_r} ((rb-a)L_r - (r^2-1)b+ra)  \\
 &\stackrel{(\star)}{=}  a L_r  - b + \frac{A_{l+1}}{L^l_r} (r(r-\frac{1}{L_r})b - (r^2-1)b+(r-L_r)a) \\
  &=  a L_r  - b + \frac{A_{l+1}}{L^l_r} ((b- L_r a) - \frac{r}{L_r}b +ra) =  a L_r  - b + \frac{A_{l+1}}{L^l_r} ((b- L_r a) + \frac{r}{L_r}(aL_r-b)) \\
  &= (1 -  \frac{A_{l+1}}{L^l_r} + \frac{r A_{l+1}}{L^{l+1}_r})  (a L_r  - b) = (\frac{L^{l+1}_r - A_{l+1} L_r+ r A_{l+1}}{L^{l+1}_r})  (a L_r  - b). \\
  \end{align*} 

\noindent We conclude with $(1)$ that $x_{l+1} L_r - y_{l+1} =  \frac{A_{l+2}}{L^{l+1}_r}(a L_r - b)$. This finishes the proof of $(i)$. The statement $(ii)$ follows also by induction on $l$. We have $A_1 = 1 < L_r$. Now assume that $A_{l} \leq L^{l}_r$. We conclude with $(1)$ that 
\[ A_{l+1} =  L^{l}_r - A_{l} L_r + r A_{l} \stackrel{(\star)}{=} L^l_r + \frac{A_l}{L_r} \leq L^l_r + L^{l-1}_r = L^{l-1}_r (L_r + 1) \leq L^{l-1}_r L^2_r = L^{l+1}_r.\]
This proves $(ii)$. \hfill $\square$

\normalsize

\section*{Acknowledgement}
\noindent I would like to thank Rolf Farnsteiner for helpful comments and suggestions which have helped to improve the structure of this article.

%%%%%%%%%%%%%%%%%%%%%%% REFERENCES %%%%%%%%%%%%%%%%%%%%%%%%%%%%%%

\end{document}